\documentclass[12pt,a4paper,openright]{article}
\title{Reflected and Doubly RBSDEs with Irregular Obstacles and a Large Set of Stopping Strategies}
       
\author{ Ihsan Arharas\thanks{Cadi Ayyad University, Department of Mathematics, Faculty of Sciences Semlalia,
B.P. 2390, Marrakesh, Morocco.
E-mail: ihsan.arharas@edu.uca.ac.ma}
      \and 
        Youssef Ouknine              
                      \thanks{Mohammed VI Polytechnic University, Africa Business School, Lot 660, Hay Moulay Rachid, Ben Guerir 43150, Morocco, and 
Cadi Ayyad University,
Department of mathematics, Faculty of Sciences Semlalia,  
B.P. 2390, Marrakesh, Morocco.
E-mail: youssef.ouknine@um6p.ma, ouknine@uca.ac.ma}
        }

\usepackage[left = 3 cm, right = 3 cm, top = 3cm, bottom = 3cm]{geometry}
\usepackage{graphicx}
\usepackage[utf8]{inputenc}
\usepackage[T1]{fontenc} 
\usepackage{tgtermes} %font type
\usepackage{relsize}
\usepackage{ntheorem}
\usepackage[colorlinks=true]{hyperref}
\usepackage{color}
\hypersetup{ colorlinks = true, linkcolor = red, citecolor= green }
\usepackage{amsmath}
\usepackage{amsmath,amssymb} %Thick police
\usepackage{mathtools}
\usepackage{cite} %retour à la ligne dans les equations
\numberwithin{equation}{section} %number the equations by sections

\usepackage{fancyhdr} %running head
\newcommand{\vertiii}[1]{{\left\vert\kern-0.25ex\left\vert\kern-0.25ex\left\vert #1 
    \right\vert\kern-0.25ex\right\vert\kern-0.25ex\right\vert}}
\DeclareMathOperator*{\esssup}{ess\,sup}
\newtheorem{theorem}{Theorem}[section]
\newtheorem{lemma}{Lemma}[section]

\newtheorem{proposition}{Proposition}[section]

\newtheorem{remark}{Remark}[section]
\newtheorem{example}{Example}
\newtheorem{definition}{Definition}[section]

\newenvironment{proof}{{\sc Proof.}}{}
\newenvironment{AMS}{}{}
\newenvironment{keywords}{}{}

\begin{document}
\maketitle
%%%%%%%%%%%%%%%%%%%%%%%%%%%
% abstract, keywords and Subject classification are optional.
%%%%%%%%%%%%%%%%%%%%%%%%%%%
\begin{abstract}
We introduce a new formulation of reflected BSDEs and doubly reflected BSDEs associated with irregular obstacles. In the first part of the paper, we consider an extension of the classical optimal stopping problem over a larger set of stopping systems than the set of stopping times (namely, the set of \textit{split stopping times}), where the payoff process $\xi$ is irregular and in the case of a general filtration. Split
stopping times are a powerful tool for modeling financial contracts and derivatives that depend on multiple conditions or triggers, and for incorporating stochastic processes with jumps and other
types of discontinuities. We show that the value family can be aggregated by an optional process $v$, which is characterized as the Snell envelope of the reward process $\xi$ over split stopping times. Using this, we prove the existence and uniqueness of a solution $Y$ to irregular reflected BSDEs. In the second part of the paper, motivated by the classical Dynkin game with completely irregular rewards considered by Grigorova et al. (2018), we generalize the previous equations to the case of two reflecting barrier processes.

\end{abstract}


\begin{AMS}
\small \textbf{AMS MSC:} \href{https://mathscinet.ams.org/msc/msc2010.html?s=60H20}{60H20}, \href{https://mathscinet.ams.org/msc/msc2010.html?s=60H30}{60H30}, \href{https://mathscinet.ams.org/msc/msc2010.html?s=65C30}{65C30}.  
\end{AMS}
%Relations with probability theory and stochastic processes37A50; 60G07; 
%general theory of processes
%Stochastic integral equations
%Applications of stochastic analysis (to PDE, etc.)
%%%%%%%%%%%%%%%%%%%%%%
% % Here is the start of the Text
%“”
%%%%%%%%%%%%%%%%%%%%%%
%\newpage

\section{Introduction}
 
\hspace*{0.5 cm} The notion of Backward stochastic differential equations (BSDEs, in short) was originally introduced
by Bismut (\cite{Bismut}, 1973) in the case of a linear driver and developed afterward by Pardoux and Peng in \cite{Pardoux and Peng}. Ever since, these equations has been widely studied in the literature due to their affiliation with many problems in different mathematical areas, we mention, among others, partial differential equations, theoretical economies, mathematical finance, stochastic optimal control, game theory, and other optimality problems. Later, El Karoui et al. \cite{El Karoui and al.} introduced the notion of backward stochastic differential equation with reflection in the case of a Brownian filtration and a continuous obstacle. This notion is simply a backward equation with the property that the first component of its solution remain greater than or equal to a given process called obstacle or barrier. The theory of reflected BSDEs for right-continuous and right-uppersemicontinuous obstacles, has been developed in a very high generality: see e.g., \cite{Hamadene, Crepey-Matoussi, Hamadene-Ouknine2003, Essaky(2008), Hamadene-Ouknine, Quenez-Sulem (2014),  Optimal stopping with, baadi2, baadi1}, however, the case of completely irregular obstacles presents considerable additional challenges. The first who dealt with this case are  Grigorova et al. in \cite{Reflected BSDEs when the obstacle is not right-continuous and optimal stopping}. Using some techniques from the optimal stopping theory (cf. \cite{El Karoui}, \cite{KobylanskiQuenez}, \cite{Maingueneau}) and some results from the  general theory of processes (cf. \cite{Probabilites et Potentiel2}), Grigorova et al. have proved existence and uniqueness of the solution to reflected BSDEs in a larger stochastic basis than the Brownian one, and where the obstacle is completely irregular. An important observation of their paper is that the solution is the value function of an optimal stopping problem in which the risk of a financial position $\xi$ is assessed by an $f$-conditional expectation $\varepsilon^f(.)$ (where $f$ is a Lipschitz driver). In this work the obstacle was considered to be an optional process. On the other hand, when the obstacle is given by a completely irregular predictable process $\xi$, and the filtration is non quasi-left continuous, Bouhadou and Ouknine \cite{Bouhadou} proved the existence of a unique predictable solution to reflected BSDEs associated with $\xi$, by means of the predictable Mertens decomposition (see \cite{Meyer_Un cours sur les integrales stochastiques}). In their setting, the predictable projection of a martingale appears in the expression of the first component of the solution, which differs from the case of reflected BSDEs associated with optional obstacles. 

Cvitanić and Karaztas (\cite{Cvitanic and Karatzas}, 1996) generalized the result of the work \cite{El Karoui and al.} to the case of two reflecting obstacles: the solution of the BSDE has to remain between two continuous processes $\xi$ and $\zeta$, called lower and upper obstacle, respectively. Since then, there has been a considerable development of doubly reflected BSDEs (DRBSDEs in short). The case of a not necessarily continuous obstacles were developed, among others, by Hamadène et al. \cite{Hamadene-Hassani-Ouknine}, Crépey and Matoussi \cite{Crepey-Matoussi}, Essaky et al. \cite{Essaky-Harraj-Ouknine}, Dumitrescu et al. \cite{Generalized Dynkin Games and Doubly reflected BSDEs with jumps}, and usually introduced in connection with various applications such as Dynkin games (see e.g., \cite{DG}). In \cite{Doubly Reflected BSDEs and E-Dynkin games}, Grigorova et al. (2020) studied DRBSDEs in the case of completely irregular optional obstacles, and interpreted the solution in terms of a game problem. They show the existence of a unique solution if and only if the so-called \textit{Mokobodzki's condition} holds (there exist two strong supermartingales such that their difference is between $\xi$ and $\zeta$). Recently,  Arharas et al. (2021) \cite{ArharasBouhadouOuknine} generalized the work  \cite{Bouhadou}, using tools from the general theory of stochastic processes, for instance, Gal'chouk and Lenglart formula for strong predictable semimartingales and the predictable Mertens decomposition. They formulate a notion of DRBSDEs in the predictable setting, where the obstacles are predictable processes and the filtration is non-quasi-left continuous. 

Although it is a known fact that every predictable process can be viewed as an optional process, references such as \cite{ArharasBouhadouOuknine}, \cite{Bouhadou}, \cite{Doubly Reflected BSDEs and E-Dynkin games} and \cite{Reflected BSDEs when the obstacle is not right-continuous and optimal stopping} have shown that reflected BSDEs in the optional and predictable frameworks exhibit noticeable differences. The main objective of this paper is to establish a formulation of reflected BSDEs that connects the two frameworks.
% to explore a compelling theoretical question: how can we formulate reflected and doubly reflected BSDEs in a way that allows for parallel interpretation of both the predictable and optional cases, and whether this new notion of BSDEs admits a unique solution?
% A theoretically very interesting question is: what can be the formulation of reflected and doubly reflected BSDEs if we wish to interpret in parallel the predictable and the optional case, and does this new BSDE's notion admits a unique solution? The main motivation of our work is to seek an answer to this question. 
The use of optimal stopping theory and the general theory of stochastic processes is generally known in most analyses of reflected BSDEs associated with irregular obstacles (see e.g., \cite{Reflected BSDEs when the obstacle is not right-continuous and optimal stopping, Optimal stopping with}), therefore, we tackle the later question by considering an augmentation of the classical optimal stopping problem, that is, by optimizing over a larger set of stopping strategies than the set of stopping times, the so-called set of \textit{split stopping times}. More precisely, we are interested in the following optimal stopping problem:
We are given a financial position/payoff process $\xi$. For a split stopping time $\delta=(G,\sigma)\in \mathcal{S}_0$, we define
\begin{equation} \label{vintro}
v(\delta) := \esssup_{\rho=(H,\tau) \in \mathcal{S}_{\delta}} E[ \xi_\rho \vert \mathcal{F}_\delta],
\end{equation}
where $\mathcal{S}_\delta$ denotes the set of split stopping times $\rho=(H, \tau)$, such that $\rho \geq \delta$, in the sense of; $\tau \geq \sigma$ a.s. and  $H \cap \lbrace \tau = \sigma \rbrace \subset G \cap \lbrace \tau = \sigma \rbrace$ (see Definition \ref{s.s.t}, Section \ref{s3}).  Note that with every stopping time $\tau$ can be associated the split stopping time $(\emptyset, \tau)$  (see e.g., \cite{Probabilites et Potentiel2}), thence the set of split stopping times is bigger than the one of stopping times.  

The notion of \textit{Split stopping time} can be traced back to Bismut \cite{Bismut1} who names it a “\textit{quasi-stopping time}” and resumed later in a slightly different form by Dellacherie and Meyer in \cite[Appendix 1, p. 424]{Probabilites et Potentiel2}. A generalization can be found in Maingueneau \cite{Maingueneau}. Split stopping times are a useful generalization of stopping times in probability theory and stochastic processes, particularly in the context of finance and financial modeling. One reason why split stopping times are used is that they provide a way to model financial contracts and derivatives that depend on multiple conditions or triggers. In many financial applications, the value of a contract or derivative may depend not only on the value of the underlying asset at a specific time (represented by a stopping time), but also on some other condition or trigger that must be met in order for the contract to be exercised or for the derivative to be valued. 
Split stopping times allow for the modeling of such conditions or triggers by incorporating an additional element $H$, which represents the condition or event that triggers the stopping time $\tau$. This can provide a more flexible and realistic representation of financial contracts and derivatives than is possible with a simple stopping time. Another reason why split stopping times are useful is that they can be used to model stochastic processes with jumps or discontinuities, which can occur in financial markets due to sudden changes in market conditions or events such as news announcements. Split stopping times allow for the modeling of such discontinuities by incorporating both the left and right limits of the process $\xi$ at time $\tau$, which provides a way to account for sudden changes in the value of the underlying asset or instrument being modeled. Overall, split stopping times are a powerful tool for modeling complex financial instruments and contracts, and for incorporating stochastic processes with jumps and other types of discontinuities. For instance, Bismut in \cite{Bismut1} introduced problem \eqref{vintro} in the case of a cadlag financial position process  $\xi$ and applied the paper results to a game problem over \textit{quasi-stopping times}. In the case of a non-right-continuous process $\xi$, this optimal stopping problem has been studied in \cite[p. 136]{El Karoui}, where the assumption of right-continuity of $\xi$ was replaced by the weaker assumption of right-uppersemicontinuity (r.u.s.c.). Recently, Marzougue \cite{Marzougue} considered a formulation of problem \eqref{vintro} closest to ours, where the reward is a ladlag positive process. Using the definition of split stopping times given in \cite{Maingueneau}, he studied the following problem: For each stopping time $\theta$,
\begin{align*}
Y(\theta) := \esssup_{\rho=(H^-,H,H^+,\tau) \in \mathcal{G}_{\theta}} E[ \xi_\rho \vert \mathcal{F}_\theta], 
\end{align*}
where $\mathcal{G}_{\theta}$ is the set of split stopping times $ \rho=(H^-,H,H^+,\tau)$, such that $\tau \geq \theta$. Nevertheless, our model is more general in two aspects: (i) the value family is indexed by split stopping times; (ii) the supremum in \eqref{vintro} is taken over the set of split stopping times greater than or equal to $\delta$, in the sense of the order defined above. We think such a generalization is interesting from theoretical point of view. From applied point of view, in many cases, the value of the process at a stopping time depends on the information available up to that time. For example, in the context of a stock market, the value of a stock at the time of sale depends on the information available to the investor up to the time of the sale. In order to find the optimal time to sell the stock, one needs to take this information into account. By using split stopping times, we are able to explicitly specify the information available up to the stopping time as a subset $H$ of the information available up to the stopping time, which is represented by the $\sigma$-algebra $\mathcal{F}_{\tau^-}$. This allows us explicitly incorporate the available information when computing the value function for an optimal stopping problem, and can help us find the optimal split stopping time and decision rule. Without the use of split stopping times, it may be difficult to specify and incorporate the available information in a clear and consistent manner.

In the present paper, we consider problem \eqref{vintro} in the case of a general filtration and without making any regularity assumptions on $\xi$. For simplification purposes, we suppose that the process $\xi$ is left-limited. We highlight that the results of this paper remain valid in the case where $\xi$ do not have left limits. This can be done by considering the process $(\bar{\xi}_t)_{t \in [0,T]}:= (\limsup\limits_{s \uparrow t, s<t} \xi_s)_{t \in [0,T]}$ instead of the process $(\xi_{t^-})_{t\in [0,T]}$ (see e.g., \cite{Doubly Reflected BSDEs and E-Dynkin games}).
Our first goal is to provide some properties of the value function family $v$ when $\xi$ is a ladlag positive process. To this purpose, some new notions for families of random variables indexed by split stopping times are introduced. Using some techniques and approaches from
the optimal stopping theory (cf. \cite{Neveu}, \cite{KobylanskiQuenez}, \cite{Maingueneau}) and some results from the  general theory of processes (cf. \cite{Probabilites et Potentiel2}), we show that the value family $(v(\delta), \delta \in\mathcal{S}_0)$ can be aggregated by a unique cadlag optional process  $(v_t)_{t\in [0,T]}$. We characterize the value process $(v_t)_{t\in [0,T]}$ as the Snell envelope of $\xi$ over split stopping times, namely, the smallest strong supermartingale greater than or equal to $\xi$. Then, we give a Mertens decomposition and we obtain some local properties of the value function $(v_t)$. Inspired by the work \cite{KobylanskiQuenez}, we underline that our results can be generalized to the case where the reward is given by a family $(\xi(\delta), \delta \in\mathcal{S}_0)$ of nonnegative random variables indexed by split stopping times. As it is mentioned in \cite{KobylanskiQuenez}, this setup is more general than the setup of processes.

An interesting question now is: using the previous results to the case where, in addition to the reward process $\xi$, there is an additional running reward process $g\in\mathcal{S}^2$, can we show that the  smallest strong supermartingale $(v_t)$ of problem \eqref{vintro} is the solution of the reflected BSDE with lower obstacle $\xi$? That is to say, can we find a new formulation of Skorohod reflecting conditions that characterizes this solution? As a second contribution of this paper, we give a generalized formulation of Skorohod reflecting conditions (see \eqref{sko1'}, \eqref{sko2'}) and then characterize the above value process as the unique solution of the following reflected BSDE: 
\begin{align} \label{RBSDEintro}
Y_\delta=\xi_{\rho^T} + \int_\delta^{\rho^T}  g(s,Y_s,Z_s) ds &- \int_\delta^{\rho^T} Z_s dW_s - (M_{\rho^T}-M_{\delta}) +  A_{\rho^T}-A_\delta \nonumber \\
&+ B_{{\rho^T}^-}-B_{\delta^-} \,\,\, \text{a.s.}, \quad \text{for all $\delta \in \mathcal{S}_0$,} 
\end{align} 
where $T$ is a fixed finite time horizon, and $\rho^T$ is the terminal split stopping time $(\emptyset, T)$. The important point to note here is that the formulation \eqref{RBSDEintro} generalized the predictable and the optional case. If we consider $\delta=(\emptyset, \tau)$, where $\tau$ runs through the
set of stopping times $\mathcal{T}_0$, we recover the usual formulation of reflected BSDEs associated with irregular optional obstacles given in \cite{Optimal stopping with}. Also, recall that, if $\tau$ is a predictable stopping time, we can associate with it the split stopping time $(\Omega, \tau)$, which can be denote by $\tau^-$ (see \cite[Appendix 1, p. 424]{Probabilites et Potentiel2}). Therefore, by associating $T$ with the split terminal time $\rho^T = (\Omega, T )$ and taking $\delta=(\Omega, \tau)$, where $\tau$ runs through the set of predictable stopping times $\mathcal{T}^p_0$, equation \eqref{RBSDEintro} coincides with the formulation of reflected BSDEs for predictable barriers given in \cite{Bouhadou}.

As a third contribution of this paper, we generalize the previous equations with the objective to give a new characterization of the problem of doubly RBSDEs in the case where the barriers do not satisfy any regularity assumption, and where the filtration is general. Motivated by ideas of \cite{Doubly Reflected BSDEs and E-Dynkin games}, under an extended type of \textit{Mokobodzki's condition}, we show the existence and uniqueness of the solution. In the proof of our result, we will use a Picard iteration method to show the existence of a solution when the driver $g$ does not depend on the solution, and for a general driver we prove the existence and uniqueness using a contaction method. Grigorova et al. \cite{Doubly Reflected BSDEs and E-Dynkin games} were the first to deal with BSDEs with two reflecting barriers that are not right-continuous. At this point we would like to mention the use of split stopping times in their context. For completely irregular payoffs, Grigorova et al. observe that the solution is characterized in terms of the value of a generalized $\mathcal{E}^f$-Dynkin game, namely, a game problem over split stopping times with (non-linear) $f$-expectation, where $f$ is the driver of the doubly reflected BSDE. Precisely, let $\theta$ be a stopping time in $\mathcal{T}_0$. For each split stopping times $\rho=(H,\tau)$ s.t. $\tau \geq \theta$ a.s. and  $\delta=(G,\sigma)$ s.t. $\sigma \geq \theta$ a.s., the criterium is given by
\begin{align*}
\mathcal{J}_{\rho,\delta}:=  \mathcal{E}_{\theta, \tau \wedge \sigma}\Big[\xi_\rho \mathbf{1}_{\tau \leq \sigma} + \zeta_\delta  \mathbf{1}_{\sigma < \tau} \Big].
\end{align*}
In view of this, we believe that our new formulation of doubly RBSDEs can open a way towards further developments in this Dynkin game problem.

The rest of the paper is organized as follows: Section \ref{s2} is devoted to notation, definitions and some preliminary results. In section \ref{s3}, we present the notion of \textit{split stopping time}, which plays a crucial role in this paper. Next, we state our first main problem of the reflected BSDE, in Definition \ref{def2}, for one left limited obstacle. Subsection \ref{sect3.1} is devoted to study our optimal stopping problem \eqref{vintro}. We first provide some properties of the value function $v$ and the strict value function $v^+$ when $\xi$ is a ladlag process. Therefore, we show that the value family $(v(\delta), \delta \in\mathcal{S}_0)$ can be aggregated by a unique ladlag optional process  $(v_t)_{t\in [0,T]}$.  Then, we give some general resuls such as Mertens decomposition of the value process and Skorokhod conditions satisfied by the associated non decreasing processes. In subsection \ref{sub 3.2}, we characterize the value process
of problem \eqref{vintro} in terms of the solution of a Reflected BSDE associated with a completely irregular obstacle in general filtration, and with a driver $g$ which does not
depend on the solution. Using a fixed point argument, we prove existence and uniqueness of the solution for a general Lipschitz driver. Section \ref{s4} is devoted to provide results on doubly reflected BSDEs associated with a Lipschitz driver and left limited barriers $(\xi,\zeta)$; in particular, we show existence and uniqueness of the solution of this equation.

\section{Notations and definitions} \label{s2}

\hspace*{0.5cm} Let $T > 0$ be a fixed positive real number. Throughout this paper we are given a probability space $(\Omega,\mathcal{F},P)$ equipped with a right-continuous complete filtration $\mathbb{F}=\lbrace\mathcal{F}_t:t\in[0,T]\rbrace$. Let $W$ be a one-dimensional $\mathbb{F}$-Brownian motion. We denote by $\mathcal{P}$ (resp. $\mathcal{O}$) the predictable (resp. optional) $\sigma$-algebra on $\Omega \times [0,T]$. Moreover, we denote by $\mathcal{T}_0$ the set of all stopping times $\tau$ with values in $[0,T]$, and by $\mathcal{T}^p_0$ the set of all predictable stopping times $\tau$ with values in $[0,T]$. More generally, for a given stopping time $S$ in $\mathcal{T}_0$, we denote by $\mathcal{T}_S$ (resp. $\mathcal{T}_{S^+}$) the set of stopping times $\tau$ in $\mathcal{T}_0 $ such that $S \leq \tau$ a.s. (resp. $S<\tau$ a.s. on $\lbrace S<T \rbrace$ and $\tau =T$ a.s. on $\lbrace S=T \rbrace$).\\

 	We also use the following notations:
\begin{enumerate}
\item[•] $L^2(\mathcal{F}_{T})$ is the set of random variables $\xi$ which are $\mathcal{F}_{T}$-measurable and square-integrable. 
\item[•] $\mathbf{H}^2$ is the set of real-valued predictable processes $\phi$ with $$\Vert \phi \Vert^2_{\mathbf{H}^2} := E \left[ \int_0^T \vert \phi_t \vert^2 dt \right] <\infty.$$ 
\item[•] $\mathbf{S}^{2}$ is the vector space of real-valued optional (not necessarily cadlag) processes $\xi$ such that 
$$ \vertiii{\xi}^2_{\mathbf{S}^{2}} := E[ \esssup_{\tau \in \mathcal{T}_0} \vert\xi_{\tau} \vert^2 ]<\infty.$$ 
 The mapping $\vertiii{.}_{\mathbf{S}^{2}}$ is a norm on the space $\mathbf{S}^{2}$. Moreover, the space $\mathbf{S}^{2}$ endowed with this norm is a Banach space. This follows by using similar arguments as in the proof of Proposition 2.1 in ~\cite{Reflected BSDEs when the obstacle is not right-continuous and optimal stopping}.

\item[•] $\mathbf{M}^2$ is the set of square integrable martingales $M=(M_t)_{t\in [0,T]}$ with $M_0=0$. \\
We can endow $\mathbf{M}^2$ with the norm $$ \Vert M \Vert_{\mathbf{M}^2} := E \lbrace M_T^2 \rbrace^{\frac{1}{2}}.$$
This space equipped with the scalar product $$(M,N)_{\mathbf{M}^2}=E [ M_T N_T ] =E [<M,N>_T ]= E[ [M,N]_T ], \,\,\,\,  M,N \in \mathbf{M}^2,$$
 is an Hilbert space. 
\item[•] $\mathbf{M}^{2,\bot}$ is the subspace of martingales $N \in \mathbf{M}^2$ satisfying $<N,W>_. =0$.
\end{enumerate}
\begin{remark}
The condition $<N,W>_. =0$ expresses the orthogonality of $N$ (in the sense of the scalar product $(.,.)_{\mathbf{M}^2}$) with respect to all stochastic integrals of the form $\int_0^. h_s dW_s$, where $h \in \mathbf{H}^2$ (cf. e.g., ~\cite[IV. 3, Lemma 2, p. 180]{Protter}). 
\end{remark} 
\hspace*{0.5cm} Let $\beta>0$. For $\phi \in \mathbf{H}^{2}$, we define $\Vert \phi \Vert_{\beta}^2 := E[\int_0^T e^{\beta s} \xi_s^2 ds]$. For $\xi \in \mathbf{S}^{2}$, we define $ \vertiii{\xi}^2_{\beta} := E[ \esssup_{\tau \in \mathcal{T}_0} e^{\tau \beta} \vert\xi_{\tau} \vert^2 ]$. For $M \in \mathbf{M}^2$, $\Vert M \Vert^2_{\mathbf{M}^2_\beta} := E(\int_0^T e^{s\beta} d[M]_s)$. 
Note that $\Vert \phi \Vert_{\beta}^2$ (resp. $\vertiii{\xi}^2_{\beta}$, $\Vert M \Vert^2_{\beta, \mathbf{M}^2}$) is a norm on $\mathbf{H}^{2}$ (resp. $\mathbf{S}^{2}$, $\mathbf{M}^2$) equivalent to the norm  $\Vert \xi \Vert^2_{\mathbf{H}^2}$ (resp. $\vertiii{\xi}^2_{\mathbf{S}^{2}}$,  $\Vert M \Vert_{\mathbf{M}^2}^2$). We write $K_\beta^2$ for the space $\mathbf{S}^{2} \times  \mathbf{H}^{2}$ equipped with the norm: $\Vert (Y,Z)\Vert^2_{K_\beta^2}:= \vertiii{Y}^2_{\beta} + \Vert Z \Vert_{\beta}^2 $, for $(Y,Z) \in \mathbf{S}^{2} \times \mathbf{H}^{2}$. Since $(\mathbf{S}^{2,p},\vertiii{.}^2_\beta )$ and $(\mathbf{H}^2, \Vert .\Vert_\beta^2 )$ are Banach spaces, $\mathbf{K}^2_\beta$ is also a Banach space. \\

For a ladlag process $\phi$, we denote by $\phi_{t+}$ and $\phi_{t-}$ the right-hand and left-hand limit of $\phi$ at time $t$. We denote by $\Delta \phi_t := \phi_t - \phi_{t-}$ the size of left jump of $\phi$ at time $t$ (with the convention $\phi_{0^-}=\phi_0$), and by $\Delta_+ \phi_t := \phi_{t+} - \phi_t$ the size of right jump of $\phi$ at time $t$.\\

The following orthogonal decomposition property of martingales in $\mathbf{M}^2$ can be found in  ~\cite[Chapter III, Lemma 4.24, p. 185]{Jacod and Shiryaev}:
\begin{lemma} \label{orthogonal decomposition}
For each $M\in \mathbf{M}^2$, there exists a unique couple $(Z,N)\in \mathbf{H}^2 \times \mathbf{M}^{2,\bot} $ such that
\begin{equation}
M_t= \int_0^t Z_s dW_s + N_t, \,\,\,\,  0\leq t\leq T \,\,\,\, a.s.
\end{equation}
\end{lemma}

\begin{definition}(Strong supermartingale)
An optional process $Y=(Y)_{t\in [0,T]}$ is said to be a strong supermartingale if 
\begin{enumerate}
\item For every bounded stopping time $\tau$, $Y_\tau$ is integrable.
\item For every pair of stopping times $S$, $\tau$ such that $S \leq \tau$,
 \begin{equation}
 Y_{S} \geq E[Y_\tau \vert \mathcal{F}_{S}] \,\,\,\, a.s.
\end{equation}   
\end{enumerate}
\end{definition}

 The following Theorems can be found in ~\cite[Theorem 86, p. 220; Theorem 53, p. 187]{Probabilites et Potentiel1}.
 
\begin{theorem}(Section theorem) \label{section}
Let $X=(X_t)$ and $Y=(Y_t)$ be two optional (resp. predictable) processes. If for every finite
stopping time $\tau$ one has, $X_\tau = Y_\tau $ a.s., then the processes $(X_t)$ and $(Y_t)$ are indistinguishable.
\end{theorem}

\begin{theorem}
For a stopping time $\tau$ in $\mathcal{T}_0$ and a measurable set $A$. We denote by $\tau_A$ the random variable defined by
\begin{equation}
\tau_A(\omega)= \tau(\omega)\quad \text{if} \quad \omega \in A, \quad \tau_A(\omega)=T \quad \text{if} \quad \omega \in A^c.
\end{equation}
 $\tau_A$ is a stopping time if and only if $A \in \mathcal{F}_\tau$.
\end{theorem}

\begin{definition}(Driver, Lipschitz driver). A function $g$ is said to be a driver if 
\begin{enumerate}
\item[(i)] $g:\Omega \times [0,T] \times \mathbb{R}^2 \rightarrow \mathbb{R}$\\
         $(\omega,t,y,z)  \mapsto g(\omega,t,y,z)$ is $\mathcal{P} \otimes \mathcal{B}(\mathbb{R}^2)-$ measurable,
\item[(ii)]$g(.,0,0) \in \mathbf{H}^2$.
\end{enumerate}
A driver $g$ is called a Lipschitz driver if moreover there exists a constant $K> 0$ such that $dP \otimes dt$-a.s., for each $(y_1,z_1)\in \mathbb{R}^2$, $(y_2,z_2)\in \mathbb{R}^2$, 
$$\vert g(w,t,y_1,z_1)-g(w,t,y_2,z_2) \vert \leq K ( \vert y_1-y_2 \vert + \vert z_1-z_2 \vert).$$
\end{definition}

We recall the definition of mutually singular random measures associated with nondecreasing cadlag predictable processes from ~\cite[Definition 2.3., p. 5]{Generalized Dynkin Games and Doubly reflected BSDEs with jumps}. 

\begin{definition}
Let $A=(A_t)_{0 \leq t\leq T}$ and $A'=(A'_t)_{0 \leq t\leq T}$ be two real-valued predictable nondecreasing cadlag processes with $A_0=0$, $A'_0=0$, $E[A_T]<\infty$ and $E[A'_T]<\infty$. We say that the random measures $dA_t$ and $dA'_t$ are \textit{mutually singular} (in a probabilistic sense), and we write $dA_t \perp dA'_t$, if there exists $D\in \mathcal{O}$ such that: 
\begin{equation}
E \left[ \int_0^T \textbf{1}_{D^c} dA_t \right]= E \left[ \int_0^T \textbf{1}_{D} dA'_t \right]=0,
\end{equation}
which can also be written as $ \int_0^T \textbf{1}_{D^c_t} dA_t= \int_0^T \textbf{1}_{D_t} dA'_t  $ a.s., where for each $t\in [0,T]$, $D_t$ is the section at time $t$ of $D$, that is, $D_t:= \lbrace \omega\in \Omega, (\omega,t)\in D \rbrace$. 
\end{definition}

\begin{definition}(Admissible obstacles).
 Let $\xi=(\xi_t)_{t\in[0,T]}$ and $\zeta=(\zeta_t)_{t\in[0,T]}$ be two processes in $\mathbf{S}^{2}$, such that $\xi_t \leq \zeta_t$, $0\leq t \leq T$ a.s. and $\xi_T = \zeta_T$ a.s. A pair of processes $(\xi,\zeta)$ satisfying the previous properties will be called a pair of admissible obstacles, or a pair of admissible barriers.
\end{definition}

\section{Optimal Stopping Problem Over Split Stopping Times and Reflected BSDEs}\label{s3}

\hspace*{0.5 cm} Let $T > 0$ be a fixed terminal time (as before). The following notion was first introduced by Bismut \cite{Bismut1} under the name “quasi-stopping times”, in a different form and resumed by Dellacherie and Meyer in \cite{Probabilites et Potentiel2}.
From practical point of view, it can be seen as a larger set of stopping strategies than the set of stopping times $\mathcal{T}_0$. The definitions below can be found in \cite{Probabilites et Potentiel2}. 
% then generalized by Maingueneau in \cite{Maingueneau}. 
  %stopping strategies than the set of stopping times $\mathcal{T}_0$. The definition below can be found in \cite{Probabilites et Potentiel2}.

\begin{definition} (\textbf{Split stopping time}) \label{s.s.t}
A split stopping time (abbreviated s.s.t.) is an ordered pair $\rho = (H, \tau)$, where $\tau$ is a stopping time in $\mathcal{T}_0$ and $H$ an element of $ \mathcal{F}_\tau$ such that $\tau_H$ is a predictable stopping time (hence $H$ in fact belongs to $\mathcal{F}_{\tau^-}$). 
$\rho$ is called bounded,  finite, ... if $\tau$ is bounded, finite, ...
\end{definition}
For a strong supermartingale or more generally a process $X$ which a.s. admits left limits and a stopping system $\rho = (H, \tau)$, we set
\begin{equation}\label{Xrho}
X_\rho :=  X_{\tau^-} \mathbf{1}_{H} + X_\tau \mathbf{1}_{H^c},
\end{equation}
with the convention that $X_{0^-}=X_0$. If $X$ is optional, $X_\rho$ is measurable with respect to the $\sigma$-algebra $\mathcal{F}_\rho$, formed by the $A \in \mathcal{F}_\tau$ such that $A \cap H \in \mathcal{F}_{\tau^-}$. \\
Let $\rho=(H, \tau)$ and $\delta=(G, \sigma)$ be two split stopping times; we write $\rho \geq \delta$ (resp. $\rho > \delta$) if $\tau \geq \sigma$ a.s. (resp. $\tau > \sigma$ a.s. on $\lbrace \sigma <T \rbrace$ and $\tau = \sigma$ a.s. on $\lbrace \sigma = T \rbrace$), and if $H \cap \lbrace \tau = \sigma \rbrace \subset G \cap \lbrace \tau = \sigma \rbrace$. $ $ It is easily seen that $\mathcal{F}_\delta \subset \mathcal{F}_\rho$.\\

In real-world applications, the event $H$ in Definition \ref{s.s.t} represents the available information up to and including the stopping time $\tau$. To illustrate this concept, let us consider an example in the context of a stock market. 
\begin{example} 
     Suppose we are modeling the price of a stock over time using a stochastic process $X$. A stopping time $\tau$ in this context could represent the time at which an investor decides to sell the stock. The event $H$ could represent the information available to the investor up to the time of the sale, such as the current market conditions, news about the company, or other factors that could influence the stock price.

In this case, a split stopping time $(H, \tau)$ would represent the situation where the investor has access to information up to the time of the sale, and can make a decision based on that information. The left limit $X_{\tau^-}$ of the stock price process at the stopping time $\tau$ would represent the value of the stock just before the sale, and $X_\tau$ would represent the value of the stock at the time of the sale. The concept of a split stopping time allows us to capture the idea that the decision to sell the stock may be based on both past and present information, and that the value of the stock at the time of the sale may be influenced by both factors. The definition of $X_\rho = X_{\tau^-} \mathbf{1}_{H} + X_\tau \mathbf{1}_{H^c}$ then allows us to compute the value of the stock at the time of the sale, taking into account the available information up to that point.
\end{example}

Let now $H^T $ be a fixed set in $\mathcal{F}_{T^-}$. We define $\rho^T$ by the split terminal time $(H^T, T)$.
We denote by $\mathcal{S}_0$ the set of all split stopping times $\delta=(G, \sigma)$ such that $\delta \leq \rho^T$, i.e. $\sigma  \leq T$ a.s. and $H^T \cap \lbrace  \sigma = T \rbrace \subset G \cap \lbrace  \sigma = T \rbrace$. More generally, for a given split stopping time $\delta$ in $S_0$, we denote by $ \mathcal{S}_\delta$ (resp. $S_{\delta+}$) the set of split stopping times $\rho=(H, \tau) \in \mathcal{S}_0$ such that $\delta \leq \rho$ (resp. $\delta < \rho$) in the above sense. We denote also by $\mathcal{S}^p_0$ the set of split stopping times $\rho=(H,\tau)\in \mathcal{S}_0$ such that $\tau \in \mathcal{T}^p_0$.
\begin{remark} \label{rmk}
\begin{enumerate}
\item Every stopping time $\tau \in \mathcal{T}_0$ can be associated with the split stopping time
$\rho=(\emptyset, \tau)$ which we identify with $\tau$. In this particular case, we have $X_\rho = X_\tau$, so the notation (\ref{Xrho}) is consistent. We see that the notion of split stopping time generalizes that of a stopping time (in the usual sense). 

\item If $\tau$ is predictable, we can associate
with it the split stopping time $(\Omega, \tau)$, which we denote by $\tau^-$. 
\end{enumerate}
\end{remark}

The following theorem can be found in \cite{Probabilites et Potentiel2}.
\begin{theorem} \label{str}
Let $X$ be an optional strong supermartingale. Let $\rho=(H,\tau)$ and $\delta=(G,\sigma)$ be two split stopping times such that $\rho \leq \delta$. Then 
\begin{equation}
X_\rho \geq \, E[X_\delta \vert  \mathcal{F}_\rho] \,\,\,\, \text{a.s.}
\end{equation}
with equality if $X$ is an optional strong martingale.
\end{theorem}

\hspace*{0.5 cm} We define now reflected BSDEs, for which the solution is constrained to stay
above a given ladlag process $\xi$ called obstacle or barrier.  %In what follows $\rho^T$ denotes the terminal split stopping time $(\emptyset ,T)$ associated with $T$. 

\begin{definition}\label{def2} Let $g$ be a driver. Let $\xi=(\xi_t)_{t\in[0,T]}$ be a ladlag process in $\mathbf{S}^{2}$.
A process $(Y,Z,M,A,B) \in \mathbf{S}^{2} \times \mathbf{H}^2 \times \mathbf{M}^{2,\perp} \times (\mathbf{S}^{2})^2$ is said to be solution to the reflected BSDE with (lower) barrier $\xi$ and driver g, if for all $\delta \in \mathcal{S}_0$
%parameters $(0,\xi)$, where $0$ is the driver and $\xi$ is a process in $\mathbf{S}^{2,p}$, if
\begin{align} \label{RBSDE}
Y_\delta=\xi_{\rho^T} + \int_\delta^{\rho^T}  g(s,Y_s,Z_s) ds &- \int_\delta^{\rho^T} Z_s dW_s - (M_{\rho^T}-M_{\delta}) +  A_{\rho^T}-A_\delta \nonumber \\
&+ B_{{\rho^T}^-}-B_{\delta^-} \quad \text{a.s.}, 
\end{align} 
 %\,\, for all}\,\,\, \delta \in \mathcal{S}_0, \,\, \delta \leq \rho^T, 
\begin{equation}
Y_\delta \geq \xi_\delta \quad \text{a.s.}
\end{equation}
The process $A$ is a nondecreasing right-continuous predictable process with $A_0 = 0$, $ E(A_{\rho^T}) < \infty$ such that:
\begin{align}
\int_0^{\rho^T} \textbf{1}_{\lbrace Y_{t^- }> \xi_{t^-} \rbrace} dA^c_t=0 \,\,\, \text{a.s. and     }  (Y_{\delta^-} - \xi_{\delta^-})(A^d_\delta - A^d_{\delta^-})=0 \,\,\, \text{a.s.} \text{ for all $\delta \in \mathcal{S}^p_0$},  \label{sko1'}
\end{align}
$B$ is a nondecreasing right-continuous adapted purely discontinuous process with $B_{0^-}=0$, $  E(B_{\rho^T}) < \infty$   and such that 
\begin{equation} \label{sko2'}
(Y_\delta-\xi_\delta)(B_\delta -B_{\delta^-})=0 \,\,\, \text{a.s. for all $\delta \in \mathcal{S}_0$.} 
\end{equation}
Here $A^c$ denotes the continuous part of the nondecreasing process $A$ and $A^d$ its
discontinuous part. Conditions (\ref{sko1'}) and (\ref{sko2'}) are called minimality conditions or
Skorohod conditions. 
\end{definition}

\begin{remark} \label{reee}
The Skorohod condition for $A^d$ (resp. $B$) can be rewritten as 
\begin{align}
&(Y_{\sigma^-} - \xi_{\sigma^-})(A^d_\sigma - A^d_{\sigma^-}) \mathbf{1}_{G^c}= 0 \,\, \text{a.s. for all $\delta=(G,\sigma) \in \mathcal{S}^p_0$} \label{equivalent Ad} \\ 
&(\text{resp.} \,\,\, (Y_\sigma-\xi_\sigma)(B_\sigma -B_{\sigma^-})\mathbf{1}_{G^c} =0 \,\,\, \text{a.s. for all}\,\,\, \delta=(G,\sigma) \in \mathcal{S}_0).
\end{align}
Indeed, recall that if $X$ is an optional left-limited process, the process $\bar{X}=(\bar{X}_t):=(X_{t^-})$ is predictable and we set for every $\delta=(G,\sigma) \in \mathcal{S}^p_0$,
\begin{align*}
X_{\delta^-} = \bar{X}_{\delta} &:= \bar{X}_{\sigma^-} \mathbf{1}_G + \bar{X}_{\sigma} \mathbf{1}_{G^c} \\
&= X_{\sigma^-} \mathbf{1}_G + X_{\sigma^-} \mathbf{1}_{G^c}= X_{\sigma^-}.
\end{align*}
Now since the processes $\xi$, $Y$ and $A$ are ladlag we get for every $\delta=(G,\sigma) \in \mathcal{S}^p_0$,  $Y_{\delta^-}=Y_{\sigma^-}$, $\xi_{\delta^-}=\xi_{\sigma^-}$, $A^d_{\delta^-}=A^d_{\sigma^-}$ and 
 \begin{align*}
 (Y_{\delta^-} - \xi_{\delta^-})(A^d_\delta - A^d_{\delta^-})= (Y_{\sigma^-} - \xi_{\sigma^-})(A^d_\delta - A^d_{\sigma^-}).
 \end{align*}
One can see that on $G$, $A^d_\delta:= A^d_{\sigma^-}$ thus $(Y_{\sigma^-} - \xi_{\sigma^-})(A^d_\delta - A^d_{\sigma^-})=0$ a.s. and on $G^c$, $A^d_\delta:= A^d_{\sigma}$ hence $(Y_{\delta^-} - \xi_{\delta^-})(A^d_\delta - A^d_{\delta^-})= (Y_{\sigma^-} - \xi_{\sigma^-})(A^d_\sigma - A^d_{\sigma^-})$ a.s.  Therefore, (\ref{equivalent Ad}) implies the Skorohod condition (\ref{sko1'}) for $A^d$. The same reasoning applies to $B$.
%  Note also that for all $\delta=(G,\sigma) \in \mathcal{S}_0$, $\delta \leq \rho^T =(\emptyset, T)$, since $\emptyset \cap \lbrace \sigma =T \rbrace \subset G \cap \lbrace \sigma =T \rbrace$. 
\end{remark}

\begin{remark}
We note that, by the section theorem \ref{section}, a process $(Y,Z,M,A,B) \in \mathbf{S}^{2} \times \mathbf{H}^2 \times \mathbf{M}^{2,\perp} \times (\mathbf{S}^{2})^2$ satisfies (\ref{RBSDE}) in the above definition, if and only if, a.s. for all $t\in [0,T]$,
\begin{align} \label{RBSDE*}
Y_t=\xi_{\rho^T} + \int_t^{^T}  g(s,Y_s,Z_s) ds - \int_t^{T} Z_s dW_s - (M_{\rho^T}-M_{t}) +  A_{\rho^T}-A_t + B_{{^{\rho^T}}^{-}}-B_{t^-}. 
\end{align} 
In the particular case where $\rho^T:=(\emptyset, T)$, the process $(Y,Z,M,A,B) \in \mathbf{S}^{2} \times \mathbf{H}^2 \times \mathbf{M}^{2,\perp} \times (\mathbf{S}^{2})^2$ satisfies (\ref{RBSDE}), if and only if, a.s. for all $t\in [0,T]$,
\begin{align*} 
Y_t=\xi_{T} + \int_t^{^T}  g(s,Y_s,Z_s) ds - \int_t^{T} Z_s dW_s - (M_{T}-M_{t}) +  A_{T}-A_t + B_{{^{T}}^{-}}-B_{t^-}. 
\end{align*} 
\end{remark}
\begin{remark}
If we rewrite Eq. \eqref{RBSDE}  forwardly, we obtain $\Delta B_\delta = -(Y_{\delta^+} - Y_\delta)$ a.s. for all $\delta \in \mathcal{S}_0$, hence $ Y_\delta \geq Y_{\delta^+}$ a.s. From the Skorokhod condition (\ref{sko2'}), we derive $(Y_\delta - \xi_\delta )(Y_{\delta^+} - Y_\delta ) = 0$ a.s. for all $\delta \in \mathcal{S}_0$. This, together with $Y_\delta \geq \xi_\delta$ a.s. and $ Y_\delta \geq Y_{\delta^+}$ a.s. yields $$Y_\delta= Y_{\delta^+} \vee \xi_\delta \quad \text{a.s. for all} \quad \delta \in \mathcal{S}_0.$$
\end{remark}

\subsection{Optimal Stopping Problem Over Split Stopping Times \label{sect3.1}} 

\hspace{0.5 cm} In this subsection, we revisit the notion of Snell envelope over split stopping times. To this end, we present some preliminary results on the value families $v$ and $v^+$ indexed by split stopping times when the reward is given by a ladlag process and in the case of a general filtration.

 Let us first introduce the definition of an admissible family of random variables indexed by split stopping times in $\mathcal{S}_0$.

\begin{definition}(\textbf{Admissible family indexed by s.s.t.}) \label{admissible}
We say that a family $U = (U (\delta), \delta \in \mathcal{S}_0)$ is admissible if it satisfies the
following conditions
\begin{enumerate}
\item For all $\delta \in \mathcal{S}_0$, $U(\delta)$ is a real-valued $\mathcal{F}_\delta$-measurable random variable. \label{prop1}
\item For all $\delta=(G,\sigma), \delta'=(G',\sigma') \in \mathcal{S}_0$, $U(\delta)= U(\delta')$ a.s. on $\lbrace \delta = \delta'\rbrace$. \label{prop2}
\end{enumerate}
If moreover for all $\delta \in \mathcal{S}_0$, $U(\delta)$ is square-integrable, we say that the admissible family $U$ is square-integrable. 
\end{definition}

 \begin{remark}
    Note that if  $U=(U(\rho), \rho \in \mathcal{S}_0)$ is an admissible family, then for each $\delta\in \mathcal{S}_0$ and $A \in \mathcal{F}_\delta$, the family $(U(\rho) \mathbf{1}_A, \rho \in \mathcal{S}_\delta)$ satisfies the properties \ref{prop1} and \ref{prop2} of Definition \ref{admissible} with $\mathcal{S}_0$ replaced by $\mathcal{S}_\delta$, and it is said to be $\delta$-admissible.
   \end{remark}

Let $(\xi_t)_{t\leq T}$ be a ladlag process in $\mathbf{S}^2$, called the reward process or the \textit{pay-off} process.  For each $\delta \in \mathcal{S}_0$, we define the value $v(\delta)$ at time $\delta$ by
\begin{equation} \label{v}
v(\delta) := \esssup_{\rho=(H,\tau) \in \mathcal{S}_{\delta}} E[ \xi_\rho \vert \mathcal{F}_\delta],
\end{equation}
and the strict value function $v^+(\delta)$ at time $\delta$ by
\begin{equation} \label{v+}
v^+(\delta) := \esssup_{\rho=(H,\tau) \in \mathcal{S}_{\delta^+}} E[ \xi_\rho \vert \mathcal{F}_\delta].
\end{equation}
Note that $v(\delta)=v^+(\delta)=\xi_{\rho^T}$ a.s. on $\lbrace \delta = \rho^T \rbrace$.

\begin{proposition} (Admissibility of $v$ and $v^+$) \label{v v+ admiss}
The families $v =(v(\delta), \delta \in \mathcal{S}_0)$ and $v^+ = (v^+(\delta), \delta \in \mathcal{S}_0)$ defined by (\ref{v}) and (\ref{v+}) are square-integrable admissible families.
\end{proposition}

\begin{proof} 
Let us prove the property for $(v^+(\delta), \delta \in \mathcal{S}_0)$. Thanks to the characterizations of the conditional expectation and of the essential supremum (see Neveu \cite{Neveu}), one can easily see that for each $\delta \in \mathcal{S}_0$, $v^+(\delta)$ is an $\mathcal{F}_\delta$-measurable square-integrable random variable. Let now $\delta=(G, \sigma)$ and $\delta'=(G', \sigma')$ be two split stopping times in $\mathcal{S}_0$, and set $A:= \lbrace \delta = \delta' \rbrace$. For every $\rho=(H, \tau) \in \mathcal{S}_{\delta^+}$, we set $\rho_A=(\bar{H}, \bar{\tau})$ where $\bar{H}:= H \cap A \cup H^T \cap A^c$ and $\bar{\tau} := \tau \mathbf{1}_{A} + T \mathbf{1}_{A^c}$. As $A=\lbrace  \sigma = \sigma'\rbrace \cap \lbrace G=G' \rbrace \in \mathcal{F}_\sigma \cap \mathcal{F}_{\sigma'}$, $\bar{\tau}$ is a stopping time in $\mathcal{T}_0$ and $\bar{H} \in \mathcal{F}_{\bar{\tau}}$. Moreover, $\bar{\tau}_{\bar{H}} = \tau_H \mathbf{1}_{A} + T \mathbf{1}_{A^c}$ is a predictable stopping time, hence $\rho_A$ is a split stopping time. Additionally, we have $\rho_A \in \mathcal{S}_{\delta'^+}$. Indeed, since $\rho \in \mathcal{S}_{\delta^+}$ we get $T>\bar{\tau} > \sigma'$ a.s. on both $\lbrace \sigma'<T \rbrace \cap A$ and $\lbrace \sigma'<T \rbrace \cap A^c$. Also we have $\bar{\tau} = T$ a.s on $\lbrace \sigma'=T \rbrace \cap A$ and $\lbrace \sigma'=T \rbrace \cap A^c$, thus $\bar{\tau} \in \mathcal{T}_{\sigma'+}$. On the other hand, %\big[ H \cap \lbrace \sigma'= \bar{\tau} \rbrace \cap A \big] \cup \big[ H^T \cap \lbrace \sigma'= \bar{\tau} \rbrace \cap A^c \big] \\&
\begin{align*}
\bar{H} \cap \lbrace \sigma'= \bar{\tau} \rbrace &= \big[ H \cap \lbrace \sigma= \tau \rbrace \cap A \big] \cup \big[ H^T \cap \lbrace \sigma'= T \rbrace \cap A^c \big]\\
&\subset \big[ G' \cap \lbrace \sigma= \tau \rbrace \cap A \big] \cup \big[ G' \cap \lbrace \sigma'= T \rbrace \cap A^c \big] \\
&= G' \cap \lbrace \sigma' = \bar{\tau}\rbrace, 
\end{align*}
and 
\begin{align*}
H^T \cap \lbrace \bar{\tau}= T \rbrace &= \big[ H^T \cap \lbrace \tau = T \rbrace \cap A \big] \cup \big[ H^T \cap \lbrace \bar{\tau}=T \rbrace \cap A^c \big] \\
&\subset \big[ H \cap \lbrace \bar{\tau} = T \rbrace \cap A \big] \cup \big[ H^T \cap \lbrace \bar{\tau}=T \rbrace \cap A^c \big] \\
&= \bar{H} \cap \lbrace \bar{\tau}= T \rbrace, 
\end{align*}
which proves that $\rho_A \in \mathcal{S}_{\delta'^+}$. One can see that $A\in \mathcal{F}_\delta \cap \mathcal{F}_{\delta'}$. Indeed, since $\sigma'_{G'}$ is a predictable stopping time, we have (see e.g., ~\cite[Theorem 73, p. 205]{Probabilites et Potentiel1})
\begin{align*}
\lbrace  \sigma = \sigma'\rbrace \cap \lbrace G=G' \rbrace \cap G = \lbrace  \sigma = \sigma'_{G'} \rbrace \cap \lbrace G=G' \rbrace \cap G' \in \mathcal{F}_{\sigma^-}.  
\end{align*}
Thus, $A\in\mathcal{F}_\delta$ and by symmetry of $\delta$ and $\delta'$ we get $A\in \mathcal{F}_\delta \cap \mathcal{F}_{\delta'}$. 
 Therefore, we obtain a.s. on $A$, $E[\xi_\rho \vert \mathcal{F}_\delta] =E[\xi_{\rho_A} \vert \mathcal{F}_\delta]= E[\xi_{\rho_A} \vert \mathcal{F}_{\delta'}] \leq v^+(\delta').$
By taking the esssup over $\mathcal{S}_{\delta^+}$ on both sides, we get $ \mathbf{1}_A v^+(\delta) \leq \mathbf{1}_A v^+(\delta')$. Again by symmetry of $\delta$ and $\delta'$, we obtain the converse inequality and the proof is complete. In the same manner we show that $(v(\delta), \delta \in \mathcal{S}_0)$ is admissible. 
\end{proof}

\begin{proposition}(\textbf{Optimizing sequences for $v$ and $v^+$})\label{opti_seq}
For each $\delta \in \mathcal{S}_0$, there exists a sequence of split stopping times $(\rho_n)_{n\in\mathbb{N}}$ in $\mathcal{S}_\delta$ (resp. $\mathcal{S}_{\delta^+}$), such that the sequence $(E[\xi_{\rho_n} \vert \mathcal{F}_\delta])_{n\in \mathbb{N}}$ is non-decreasing and 
$$v(\delta) \quad (\text{resp.} v^+(\delta))\,\,\, = \lim\limits_{n\rightarrow \infty} \uparrow E[\xi_{\rho_n} \vert \mathcal{F}_\delta] \quad \text{a.s.}$$
\end{proposition}

   \begin{proof} We give the proof only for $(v(\delta), \delta \in \mathcal{S}_0)$. The same reasoning applies to $(v^+(\delta), \delta \in \mathcal{S}_0)$. We fix $\delta= (G, \sigma) \in \mathcal{S}_0$ and consider the set of random variables $ E[ \xi_\rho \vert \mathcal{F}_\delta]$, where $\rho$ runs through the set of split stopping times $\mathcal{S}_\delta$, one can show that this set is closed under pairwise maximization. Indeed, let $\rho_1=(H_1,\tau_1) \in \mathcal{S}_\delta$ and $ \rho_2=(H_2,\tau_2) \in \mathcal{S}_\delta$, put $A=\lbrace E[\xi_{\rho_1} \vert \mathcal{F}_\delta] \geq E[\xi_{\rho_2} \vert \mathcal{F}_\delta]\rbrace$. One has $A \in \mathcal{F}_\delta$. Put $\tau := \tau_1 \mathbf{1}_A + \tau_2 \mathbf{1}_{A^c}$ and $H := H_1 \cap A \cup H_2 \cap A^c$. Then $H\in \mathcal{F}_\tau$, $\tau \in \mathcal{T}_\sigma$ and $\tau_H= \tau_{1_{H_1}} \mathbf{1}_A + \tau_{2_{H_2}} \mathbf{1}_{A^c}$ is predictable. Moreover,
\begin{align*}
H \cap \lbrace \tau = \sigma \rbrace &= \big[ H_1 \cap \lbrace \tau_1 = \sigma \rbrace \cap A \big] \cup  \big[ H_2 \cap \lbrace \tau_2 = \sigma \rbrace \cap A^c \big]\\
&\subset  \big[ G \cap \lbrace \tau = \sigma \rbrace \cap A \big] \cup \big[  G \cap \lbrace \tau = \sigma \rbrace \cap A^c \big] \\
&= G \cap \lbrace \tau = \sigma \rbrace,
\end{align*}
and 
\begin{align*}
H^T \cap \lbrace \tau = T \rbrace &= \big[ H^T \cap \lbrace \tau_1 = T \rbrace \cap A \big] \cup  \big[ H^T \cap \lbrace \tau_2 = T \rbrace \cap A^c \big]\\
&\subset  \big[ H_1 \cap \lbrace \tau = T \rbrace \cap A \big] \cup \big[  H_2 \cap \lbrace \tau = T \rbrace \cap A^c \big] \\
&= H \cap \lbrace \tau = T \rbrace.
\end{align*}
It follows that $\rho := (H, \tau)$ is a split stopping time that belongs to $\mathcal{S}_\delta$. Therefore,
\begin{align*}
E[\xi_\rho \vert \mathcal{F}_\delta] &= E[\xi_{\tau-} \mathbf{1}_H + \xi_\tau \mathbf{1}_{H^c} \vert \mathcal{F}_\delta] \\
&= E[\xi_{{\tau_1}^-}  \mathbf{1}_{H \cap A} + \xi_{{\tau_2}^-}  \mathbf{1}_{H \cap A^c} + \xi_{\tau_1}  \mathbf{1}_{H^c \cap A} + \xi_{\tau_2}  \mathbf{1}_{H^c \cap A^c} \vert \mathcal{F}_\delta ] \\
&= E[\xi_{\tau_1^-} \mathbf{1}_{H_1} + \xi_{\tau_1} \mathbf{1}_{H_1^c} \vert \mathcal{F}_\delta]\,\, \mathbf{1}_A + E[\xi_{\tau_2^-} \mathbf{1}_{H_2} + \xi_{\tau_2} \mathbf{1}_{H_2^c} \vert \mathcal{F}_\delta]\,\,  \mathbf{1}_{A^c}\\
&= E[\xi_{\rho_1} \vert \mathcal{F}_\delta] \,\, \mathbf{1}_A + E[\xi_{\rho_2} \vert \mathcal{F}_\delta]\,\,  \mathbf{1}_{A^c} = E[\xi_{\rho_1} \vert \mathcal{F}_\delta] \vee E[\xi_{\rho_2} \vert \mathcal{F}_\delta],
\end{align*}
which implies the stability under pairwise maximization. The proof is straightforward using a classical result on essential suprema (cf. Neveu \cite{Neveu}).
   \end{proof}

   \begin{definition}(\textbf{Supermartingale family indexed by s.s.t.}) \label{s.m.f}
   An admissible family $U:=(U(\delta), \delta \in \mathcal{S}_0)$ is said to be a supermartingale family (resp. a martingale family) if for all $\delta, \delta' \in \mathcal{S}_0$, such that $\delta \geq \delta'$,
   $$E[U(\delta) \vert \mathcal{F}_{\delta'}] \leq U(\delta') \quad \text{a.s. (resp.} \quad E[U(\delta) \vert \mathcal{F}_{\delta'}] = U(\delta') \quad \text{a.s.)}.$$    
   \end{definition}
   
   The following proposition states that the value function $v$ and the strict value function $v^+$ are both supermartingale families. 
   
   \begin{proposition} \label{snell env family} The admissible families $v =(v(\delta), \delta \in \mathcal{S}_0)$ and $v^+ = (v^+(\delta), \delta \in \mathcal{S}_0)$ are supermartingale families in the sense of Definition \ref{s.m.f}. Moreover, the value family $v =(v(\delta), \delta \in \mathcal{S}_0)$ is characterized as the Snell envelope family over split stopping times associated with $(\xi_\delta, \delta \in \mathcal{S}_0)$, that is, the smallest supermartingale family indexed by split stopping times which is greater a.s. than $(\xi_\delta, \delta \in \mathcal{S}_0)$. 
   \end{proposition} 

\begin{proof} 
We only prove the property for $v =(v(\delta), \delta \in \mathcal{S}_0)$. Similar arguments apply for $v^+ = (v^+(\delta), \delta \in \mathcal{S}_0)$. Let $\delta, \delta' \in \mathcal{S}_0$, such that $\delta \geq \delta'$. By Proposition \ref{opti_seq}, there exists a sequence of split stopping times $(\rho_n)_{n\in\mathbb{N}}$ in $\mathcal{S}_\delta$, such that $v(\delta) = \lim\limits_{n\rightarrow \infty} \uparrow E[\xi_{\rho_n} \vert \mathcal{F}_\delta]$ a.s. From this together with the monotone convergence theorem, we get
$$E[v(\delta) \vert \mathcal{F}_{\delta'}]= E[ \lim\limits_{n\rightarrow \infty} \uparrow E[\xi_{\rho_n} \vert \mathcal{F}_\delta]  \vert \mathcal{F}_{\delta'} ]= \lim\limits_{n\rightarrow \infty} \uparrow E[\xi_{\rho_n} \vert \mathcal{F}_{\delta'} ] \leq v(\delta')\quad \text{a.s.,}$$ and thus the supermartingale property of $v$ holds. What is left is to show the second statement. We see at once that $(v(\delta), \delta \in \mathcal{S}_0)$ is a supermartingale family bounding $\xi$ above, i.e. for each $\delta \in \mathcal{S}_0$, $v(\delta) \geq \xi_\delta$ a.s.  Let $(\bar{v}(\delta), \delta \in \mathcal{S}_0)$ be another supermartingale family such that for each $\delta \in \mathcal{S}_0$, $\bar{v}(\delta) \geq \xi_\delta$ a.s. Fix $\delta \in \mathcal{S}_0$. By the supermartingale property of $\bar{v}$, we have for every split stopping time $\rho \in\mathcal{S}_\delta$,
\begin{equation*}
\bar{v}(\delta) \geq E[\bar{v}(\rho) \vert \mathcal{F}_\delta] \geq E[ \xi_\rho \vert \mathcal{F}_\delta] \quad \text{a.s.}
\end{equation*}
Taking the supremum over $\rho$, we obtain $\bar{v}(\delta) \geq v(\delta)$ a.s., and the proposition follows. \\
\end{proof} 
   \begin{proposition} \label{v=xi vee v+}
  For every $\delta \in \mathcal{S}_0$,  $v(\delta) = \xi_\delta \vee v^+(\delta)$ a.s.
   \end{proposition}
   
   \begin{proof}
  Let $\delta=(G, \sigma)$ be a split stopping time in  $\mathcal{S}_0$. Take $\rho=(H,\tau) \in \mathcal{S}_\delta$, i.e. $\rho \geq \delta$.
 We first show that 
    \begin{equation}
    E[\xi_\rho \vert \mathcal{F}_\delta] \leq \xi_\delta \vee v^+(\delta) \quad \text{a.s.}
    \end{equation}
    We set $\gamma=( H,\bar{\tau}):=(H, \tau \mathbf{1}_{\lbrace \rho> \delta \rbrace} + T \mathbf{1}_{\lbrace \rho= \delta \rbrace})$. One can show that $\gamma$ belongs to $\mathcal{S}_{\delta^+}$. Indeed, since $\rho \geq \delta$, $H \in \mathcal{F}_\tau \subset \mathcal{F}_{\bar{\tau}}$ and $\lbrace \rho= \delta \rbrace \in \mathcal{F}_{\tau^-}$, $\bar{\tau}$ is a stopping time and $\bar{\tau}_H= \tau_H \mathbf{1}_{\lbrace \rho> \delta \rbrace} + T \mathbf{1}_{\lbrace \rho= \delta \rbrace}$ is a predictable stopping time in $\mathcal{T}^p_0$. Moreover, $\bar{\tau} \in \mathcal{T}_{\sigma^+}$ and hence $\gamma  \in \mathcal{S}_{\delta^+}$. Then,
    \begin{equation}\label{Eq0}
    E[\xi_\rho \vert \mathcal{F}_\delta] \mathbf{1}_{\lbrace \rho>\delta \rbrace} = E[\xi_\gamma \vert \mathcal{F}_\delta] \mathbf{1}_{\lbrace \rho> \delta \rbrace}\leq v^+(\delta) _{\lbrace \rho \geq  \delta \rbrace} \quad \text{a.s.}
    \end{equation}
   Therefore, 
   $$E[\xi_\rho \vert \mathcal{F}_\delta]= \xi_\delta \mathbf{1}_{\lbrace \rho = \delta \rbrace} +  E[\xi_\rho \vert \mathcal{F}_\delta] \mathbf{1}_{\lbrace \rho> \delta \rbrace} \leq \xi_\delta \mathbf{1}_{\lbrace \rho = \delta \rbrace} +  v^+(\delta) _{\lbrace \rho \geq  \delta \rbrace} \leq  \xi_\delta \vee v^+(\delta) \quad \text{a.s.} $$
   By taking the essential supremum over $\rho \in \mathcal{S}_\delta$, we get $v(\delta) \leq \xi_\delta \vee v^+(\delta)$ a.s.
   The other inequality follows immediately from the fact that $v(\delta) \geq v^+(\delta)$ a.s. and $v(\delta) \geq \xi_\delta$ a.s., which completes the proof.\\ 
   \end{proof}
   
   We now give a crucial property of regularity for the strict value function family, that is, the right continuity along
split stopping times in expectation (in short RCE). This property is well-known in the case of families indexed by stopping times and it is pretty near to Proposition 1.12. in Kobylanski and Quenez (cf. \cite{KobylanskiQuenez}).
   
   \begin{definition}(\textbf{Right continuous family along s.s.t. in expectation}) \label{DefRCE}
   An admissible family $U= (U(\delta), \delta \in \mathcal{S}_0)$ is said to be right continuous along
split stopping times in expectation (RCE) if for every $\delta=(G,\sigma) \in \mathcal{S}_0$ and for any sequence of split stopping times $(\delta_n)_{n\in \mathbb{N}}=(G^n,\sigma_n)_{n\in \mathbb{N}}\in \mathcal{S}_0$ such that $\delta_n \downarrow \delta$ in the sense; $\sigma_n \downarrow \sigma$ and $G^n \downarrow G$, one has $E[U(\delta)] = \lim\limits_{n \rightarrow \infty} E[U(\delta_n)]$.
   \end{definition}
   
   \begin{proposition}(RCE property for $v^+$ along s.s.t.) \label{v+RCE} Let $\xi \in \mathbf{S}^2$ be a non-negative ladlag process. The associated strict value function family $v^+ = (v^+(\delta), \delta \in \mathcal{S}_0)$ is RCE along split stopping times.
   \end{proposition}
  
  \begin{proof}
   First, note that for each $\delta=(G,\sigma) \in \mathcal{S}_0$, $E[v^+(\delta)]<\infty$. Indeed, for every $\rho=(H,\tau) \in \mathcal{S}_\delta$ we have $\xi_\rho:= \xi_{\tau^-} \mathbf{1}_H + \xi_\tau \mathbf{1}_{H^c}$. By the left-continuity of the process $(\xi_{t^-})$, we get 
  \begin{align} \label{Eq5}  
  \sup\limits_{t \in[0,T]} \vert \xi_{t^-} \vert^2= \sup\limits_{t \in \mathbb{Q}\cap [0,T] } \vert \xi_{t^-} \vert^2 \leq \esssup\limits_{\theta \in \mathcal{T}_0} \vert \xi_\theta \vert^2 \quad \text{a.s.}
    \end{align}
  Using this together with the existence of an optimizing sequence of split stopping times for $v^+(\delta)$ (see Proposition \ref{opti_seq}), it follows that 
  \begin{align}
E[v^+(\delta)]= \sup\limits_{\rho \in \mathcal{S}_{\delta^+}} E[\xi_\rho] &\leq E[\esssup\limits_{\rho \in \mathcal{S}_{\delta^+}} \vert \xi_\rho \vert^2]^{\frac{1}{2}} \label{E(v+)}  \\
 &\leq  E\big[\esssup\limits_{\theta \in \mathcal{T}_0} \vert \xi_{\theta^-} \vert^2 \mathbf{1}_H  + \esssup\limits_{\theta \in \mathcal{T}_0} \vert \xi_{\theta} \vert^2 \mathbf{1}_{H^c} \big]^{\frac{1}{2}} \nonumber \\
&\leq  E\big[\esssup\limits_{\theta \in \mathcal{T}_0} \vert \xi_{\theta} \vert^2  \big]^{\frac{1}{2}}<\infty. \nonumber
  \end{align}
  It is also easily seen that the function $\delta \rightarrow E[v^+(\delta)]$ is a non-increasing function of split stopping times, since $v^+$ is a supermartingal family. Suppose, contrary to our claim, that the family $(v^+(\delta), \delta \in \mathcal{S}_0)$ is not RCE at $\delta=(G, \sigma) \in \mathcal{S}_0$. In other words, there exists a constant $\epsilon>0$, and a sequence of split stopping times $(\delta_n)_{n\in \mathbb{N}}=(G^n,\sigma_n)_{n\in \mathbb{N}}\in \mathcal{S}_0$ such that $\delta_n \downarrow \delta$ in the sense of Definition \ref{DefRCE} and 
  \begin{equation*}
  \lim\limits_{n \rightarrow \infty} \uparrow E[v^+(\delta_n)] + \epsilon \leq E[v^+(\rho)].
  \end{equation*}
  From Eq. (\ref{E(v+)}), one can easily see that there exists $\rho' \in \mathcal{S}_{\delta^+}$, such that
  \begin{equation} \label{Eq}
   \lim\limits_{n \rightarrow \infty} \uparrow E[v^+(\delta_n)] + \frac{\epsilon}{2} \leq E[\xi_{\rho'}].
  \end{equation}
  In order to find a contradiction, suppose at first that $\sigma < T$ a.s. Therefore, $\rho'=(H',\tau') \in \mathcal{S}_{\delta^+}$  gives that $\tau'> \sigma$ a.s. and we have 
 \begin{align*}
  \lbrace \rho' > \delta \rbrace &= \Big\{ \omega: \tau'(\omega)> \sigma(\omega), H' \cap \lbrace \tau'(\omega)=\sigma(\omega)\rbrace \subset G \cap \lbrace \tau'(\omega)= \sigma(\omega) \rbrace \Big\}\\
 &= \bigcup_{n\in\mathbb{N}} \uparrow \lbrace \tau' > \sigma_n\rbrace 
 \end{align*} 
 Thus, $E[\xi_\rho']= \lim\limits_{n \rightarrow \infty} \uparrow E[\xi_{\rho'} \mathbf{1}_{\tau' >\sigma_n}]$. Using Eq. (\ref{Eq}), one can see that there exists an $n_0$ such that 
 \begin{align*}
 \lim\limits_{n \rightarrow \infty} \uparrow E[v^+(\delta_n)] + \frac{\epsilon}{4} \leq E[\xi_{\rho'} \mathbf{1}_{\tau' >\sigma_{n_0}}].
 \end{align*}
We set $\bar{\rho}=(\bar{H}, \bar{\tau})$, where $\bar{H}:= \big[ H' \cap \lbrace \tau' > \sigma_{n_0} \rbrace \big] \cup  \big[ H^T \cap \lbrace \tau' \leq \sigma_{n_0} \rbrace \big]$ and $\bar{\tau}:= \tau' \mathbf{1}_{\tau'> \sigma_{n_0}} + T \mathbf{1}_{\tau' \leq \sigma_{n_0}}$. One has $\bar{\tau} > \sigma_{n_0}$ a.s. on $\lbrace \sigma_{n_0} < T\rbrace$, 
$\bar{\tau} = T$ a.s. on $\lbrace \sigma_{n_0} = T\rbrace$ and 
\begin{align*}
\bar{H} \cap \lbrace \bar{\tau}= \sigma_{n_0} \rbrace &= \big[ H' \cap \lbrace \tau' > \sigma_{n_0} \rbrace \cap \lbrace \tau'= \sigma_{n_0} \rbrace \big] \cup  \big[ H^T \cap \lbrace \tau' \leq \sigma_{n_0} \rbrace \cap \lbrace T= \sigma_{n_0} \rbrace \big]\\
&= \emptyset \cup \big[ H^T \cap \lbrace \tau' \leq \sigma_{n_0} \rbrace \cap \lbrace T= \sigma_{n_0} \rbrace \big]\\
&\subset \big[ G^{n_0} \cap \lbrace \tau' > \sigma_{n_0} \rbrace \cap \lbrace \bar{\tau}= \sigma_{n_0} \rbrace \big] \cup  \big[ G^{n_0} \cap \lbrace \tau' \leq \sigma_{n_0} \rbrace \cap \lbrace \bar{\tau}= \sigma_{n_0} \rbrace \big]\\
&= G^{n_0} \cap \lbrace \bar{\tau}= \sigma_{n_0} \rbrace.
\end{align*}
 We have thus proved that $\bar{\rho} \in \mathcal{S}_{\delta_{n_0}^+}$. Therefore, the positivity of $\xi$ yields
$$ E[v^+(\delta_{n_0})] + \frac{\epsilon}{4} \leq  E[\xi_{\rho'} \mathbf{1}_{\tau' >\sigma_{n_0}}]\leq E[\xi_{\bar{\rho}}] \leq E[v^+(\delta_{n_0})],$$
  which is impossible.\\
  To study the general case, take $\sigma\in \mathcal{T}_0$. Since $\rho' \in \mathcal{S}_{\delta^+}$, one has $\rho'=\rho^T$ on $\lbrace \delta = \rho^T \rbrace$ and $\tau' > \sigma$ a.s. on $\lbrace \delta < \rho^T \rbrace$. Thus, $E[\xi_{\rho'}]= E[\xi_{\rho'} \mathbf{1}_{\delta <\rho^T}] + E[\xi_{\rho^T} \mathbf{1}_{\delta =\rho^T}]$ and  
  \begin{align*}
  E[\xi_{\rho'} \mathbf{1}_{\delta <\rho^T}]= \lim\limits_{n \rightarrow \infty} \uparrow E[\xi_{\rho'} \mathbf{1}_{\lbrace \delta <\rho^T \rbrace \cap \lbrace \tau' > \sigma_n \rbrace  }].
  \end{align*}
  From this and Eq. (\ref{Eq}), there exists an $n_0$ such that 
  \begin{align} \label{Eq1}
  \lim\limits_{n \rightarrow \infty} \uparrow E[v^+(\delta_n)] + \frac{\epsilon}{4} \leq E[\xi_{\rho'} \mathbf{1}_{\lbrace \delta <\rho^T \rbrace \cap \lbrace \tau' > \sigma_{n_0} \rbrace}] +  E[\xi_{\rho^T} \mathbf{1}_{\delta =\rho^T}].
  \end{align}
  Consider $\bar{\rho}:=(\bar{H}, \bar{\tau})$, where
  \begin{align*}
  \bar{H} &= \big[H' \cap \lbrace \tau'> \sigma_{n_0} \rbrace \cap \lbrace \sigma <T \rbrace  \big] \cup \big[H^T \cap \lbrace \tau'\leq \sigma_{n_0} \rbrace \cap \lbrace \sigma <T \rbrace \big] \cup \big[ H^T \cap \lbrace \sigma = T \rbrace \big], \\
  \bar{\tau} &= \tau' \mathbf{1}_{\lbrace \tau'> \sigma_{n_0} \rbrace \cap \lbrace \sigma <T \rbrace} + T \mathbf{1}_{\lbrace \tau'\leq \sigma_{n_0} \rbrace \cap \lbrace \sigma <T \rbrace} + T \mathbf{1}_{\lbrace \sigma = T \rbrace}.
  \end{align*}
  By similar arguments as in the first case, one can check that $\bar{\rho}$ is a split stopping time in $ \mathcal{S}_{\delta_{{n_0}}^+}$. Consequently,
  \begin{align*}
  E[\xi_{\rho'} \mathbf{1}_{\lbrace \delta <\rho^T \rbrace \cap \lbrace \tau' > \sigma_{n_0} \rbrace}] +  E[\xi_{\rho^T} \mathbf{1}_{\delta =\rho^T}] \leq E[\xi_{\bar{\rho}}] \leq E[v^+(\delta_{n_0})]. 
  \end{align*}
  From this and Eq. (\ref{Eq1}), we derive $ E[v^+(\delta_{n_0})] + \frac{\epsilon}{4} \leq \lim\limits_{n \rightarrow \infty} \uparrow E[v^+(\delta_n)] + \frac{\epsilon}{4} \leq  E[v^+(\delta_{n_0})]$, a contradiction. The proof of the proposition is thus complete.\\
  
  % $\bar{H}:= \big[H' \cap \lbrace \tau'> \sigma_{n_0} \rbrace \cap \lbrace \sigma <T \rbrace  \big] \cup \big[H^T \cap \lbrace \tau'\leq \sigma_{n_0} \rbrace \cap \lbrace \sigma <T \rbrace \big] \cup \big[ H^T \cap \lbrace \sigma = T \rbrace \big]$ and $\bar{\tau}:= \tau' \mathbf{1}_{\lbrace \tau'> \sigma_{n_0} \rbrace \cap \lbrace \sigma <T \rbrace} + T \mathbf{1}_{\lbrace \tau'\leq \sigma_{n_0} \rbrace \cap \lbrace \sigma <T \rbrace} + T \mathbf{1}_{\lbrace \sigma = T \rbrace}$.
  \end{proof}
   The following Lemma holds:
	\begin{lemma}
	Let $U=(U(\rho), \rho \in \mathcal{S}_0)$ be a RCE family along split stopping times. Then, for every $\delta=(G,\sigma) \in \mathcal{S}_0$ and $A\in \mathcal{F}_\delta$, the family $(U(\rho) \mathbf{1}_A, \rho \in \mathcal{S}_\delta)$ is RCE.
	\end{lemma}   
   
   \begin{proof}
Let $\delta=(G,\sigma) \in \mathcal{S}_0$. Let $(\rho_n)_{n\in \mathbb{N}}= (H^n, \tau_n)_{n\in \mathbb{N}}$ be a sequence of split stopping times such that $\rho_n \downarrow \rho$ in the sense of Definition \ref{DefRCE}. For each $n\in \mathbb{N}$, put $\bar{\rho}_n:=(H^n \cap A \cup H^T \cap A^c, \tau_n \mathbf{1}_A + T \mathbf{1}_{A^c})$, and $\bar{\rho}:= (H \cap A \cup H^T \cap A^c, \tau \mathbf{1}_A + T \mathbf{1}_{A^c})$.  Clearly, we have $\bar{\rho}_n \downarrow \bar{\rho}$.  From the RCE property of $(U(\rho), \rho \in \mathcal{S}_0)$, we get $\lim\limits_{n \rightarrow \infty} E[U(\bar{\rho}_n)]= E[U(\bar{\rho})]$, consequently $\lim\limits_{n \rightarrow \infty} E[U(\rho_n) \mathbf{1}_A ]= E[U(\bar{\rho}) \mathbf{1}_A]$. \\
  \end{proof}
   
   Let us now state the following lemma that will be used in the sequel. 
   
   \begin{lemma} \label{lemma}
   Let $\xi$ be a ladlag process in  $\mathbf{S}^{2}$. For each $\delta=(G,\sigma), \rho=(H,\tau) \in \mathcal{S}_0$, we have 
   \begin{align*}
   E[v(\rho) \vert \mathcal{F}_\delta] \leq v^+(\delta) \quad \text{a.s. on} \,\, \lbrace \rho>\delta \rbrace. 
   \end{align*}
   \end{lemma}
   
   \begin{proof}
   From Proposition \ref{opti_seq}, there exists a sequence of split stopping times $(\rho_n)_{n\in\mathbb{N}}=(H^n, \tau_n)_{n\in\mathbb{N}}$ in $\mathcal{S}_\rho$, such that the sequence $(E[\xi_{\rho_n} \vert \mathcal{F}_\delta])_{n\in \mathbb{N}}$ is non-decreasing and 
$$v(\rho) = \lim\limits_{n\rightarrow \infty} \uparrow E[\xi_{\rho_n} \vert \mathcal{F}_\rho] \quad \text{a.s.}$$
   Therefore, using the fact that $\xi \in \mathbf{S}^{2}$, by dominated convergence theorem for conditional expectation, we get a.s. on $\lbrace \rho>\delta \rbrace$
   \begin{equation}\label{Eq2}
   E[v(\rho) \vert \mathcal{F}_\delta] = E\big[ \lim\limits_{n\rightarrow \infty} \uparrow E[\xi_{\rho_n} \vert \mathcal{F}_\rho]  \vert \mathcal{F}_\delta\big]=  \lim\limits_{n\rightarrow \infty} \uparrow E[\xi_{\rho_n} \vert \mathcal{F}_\delta].
   \end{equation}
   Moreover, for each $n\in\mathbb{N}$, we have $\tau_n \geq \tau$ and $H^n \cap \lbrace \tau_n=\tau \rbrace \subset H \cap \lbrace \tau_n=\tau \rbrace$. Thus, on $\lbrace \rho>\delta \rbrace$, we have; $\tau_n \geq \tau>\sigma$ a.s. on $\lbrace \sigma<T \rbrace$, $\tau_n=\tau=T$ a.s. on $\lbrace \sigma=T \rbrace$ and $H \cap \lbrace \sigma=\tau \rbrace \subset G \cap \lbrace \sigma=\tau\rbrace$, hence 
   \begin{align*}
   H^n \cap \lbrace \tau_n=\sigma \rbrace&=  H^n \cap \lbrace \tau_n=\sigma \rbrace \cap \lbrace \sigma=T \rbrace  \cap \lbrace \sigma=\tau \rbrace\\
   &\subset H \cap \lbrace \tau_n=\sigma \rbrace \cap \lbrace \sigma=T \rbrace  \cap \lbrace \sigma=\tau \rbrace\\
   &\subset G \cap \lbrace \tau_n=\sigma \rbrace,
   \end{align*}
   which shows that on $\lbrace \rho>\delta \rbrace$, $\rho_n>\delta$. Then, from Eq. (\ref{Eq0}), we get $E[\xi_{\rho_n} \vert \mathcal{F}_\delta]\leq v^+(\delta)$ a.s. From this and Eq. (\ref{Eq2}) we obtain the expected result $E[\xi_{\rho} \vert \mathcal{F}_\delta]\leq v^+(\delta)$ a.s. on $\lbrace \rho>\delta \rbrace$.\\
   \end{proof}
   
  We will now give a regularity property which holds for any uniformly integrable supermartingale family indexed by split stopping times. We first present the definition of an admissible family that is right limited along split stopping times (RL in short). 
   
   \begin{definition}(\textbf{Right-limited family}) \label{RL}
   An admissible family $(U(\delta), \delta\in \mathcal{S})$ is said to be right limited along split stopping times (RL) if for every $\delta=(G,\sigma)\in\mathcal{S}_0$  there exists an $\mathcal{F}_\delta$-measurable random variable $U(\delta^+)$ such that, for any non-increasing sequence of split stopping times $(\delta_n)_{n\in \mathbb{N}} =(G^n,\sigma_n)_{n\in \mathbb{N}}$ such that $\delta_n \downarrow \delta$ (in the sense, $G^n \downarrow G$, $\sigma_n \downarrow \sigma$ and for each $n$, $\delta_n>\delta$), one has $U(\delta^+)= \lim\limits_{n \rightarrow \infty} U(\delta_n)$ a.s. 
   \end{definition}
   
   \begin{theorem} \label{RL theorem}
A uniformly integrable supermartingale family $(U(\rho), \rho\in \mathcal{S}_0)$, is right limited along split stopping times (RL) at any $\delta=(G, \sigma)\in\mathcal{S}_0$.
   \end{theorem}
   
   \begin{proof}
   Fix $\delta \in \mathcal{S}_0$. Let $(\delta_n)_{n\in \mathbb{N}} =(G^n,\sigma_n)\in \mathcal{S}_{\delta^+}$, such that $\delta_n \downarrow \delta$ in the sense of Definition \ref{RL}. Put $Z_n:= U(\delta_{-n})$ and $\mathcal{G}_n= \mathcal{F}_{\delta_{-n}}$ for every $n\leq 0$. The sequence $(Z_n)_{n\leq0}$ is a supermartingale relative to the non-decreasing filtration $(\mathcal{G}_n)_{n\leq0}$. Since the family $U$ is uniformly integrable, by the convergence theorem for discrete supermartingales indexed by non-positive integers, and uniformly bounded in $L^1$  (cf., Chap.V, Theorem 30, \cite{Probabilites et Potentiel2}, p. 26), there exists an integrable random variable $Z$ such that the sequence $(Z_n)_{n\leq0}$ converges a.s. and in $L^1$ to $Z$. The random variable $Z$ is $\mathcal{F}_\delta$-measurable. Indeed, we have $\delta_n \downarrow \delta$, i.e., $\sigma_n \downarrow \sigma$, $G^n \downarrow G$, and for each $n$ $\delta_n > \delta$.  Hence, for every $n\leq 0$, $\sigma_{-n} \uparrow \sigma$, and $\sigma_{-n} < \sigma$. Moreover, we have $G=\cap_{n \geq 0} \downarrow G^n$, then $G \subset G^{-n}$, for every $n\leq 0$. Thus, $\delta^{-n} < \delta$, and $\mathcal{F}_{\delta_{-n}} \subset \mathcal{F}_{\delta}$, for each $n\leq 0$. Therefore, $Z_n:=U(\delta_{-n})$ is  $\mathcal{F}_{\delta}$-measurable for all $n\leq 0$. 
Finally, since the sequence $(Z_n)_{n\leq0}:= (U(\delta_{-n}))_{n\leq0}$ converges a.s. and in $L^1$ to $Z$, and the limit a.s. preserves measurability, we conclude that $Z$ is $\mathcal{F}_{\delta}$-measurable.
   
   As such, we set $U(\delta^+)$ to be $U(\delta^+):=Z$. We are left with the task of showing that this limit does not depend on the sequence $(\delta_n)$. For this purpose, let $(\delta'_n)_{n\in \mathbb{N}} =(G'^n,\sigma'_n)\in \mathcal{S}_{\delta^+}$ be a sequence such that $\delta'_n \downarrow \delta$. Similarly, there exists a random variable $Z'$ such that $(Z'_n)_{n\leq0}$ converges a.s. and in $L^1$ to $Z'$. Our goal is to show that $Z=Z'$ a.s. \\
 \hspace*{0.5 cm}For each $n$ and each $\omega \in \Omega$, we define $\bar{\sigma}_n(\omega):= \sigma^{(n)}(\omega)$ \big(resp. $\bar{G}^n:= G^{(n)}$\big), where $\sigma^{(2n)}(\omega)\leq \sigma^{(2n-1)}(\omega) \leq  \ldots \leq \sigma^{(0)}(\omega)$ \big(resp. $G^{(2n)} \subset G^{(2n-1)},  \ldots, G^{(0)}$\big) are the reordered terms of the sequence $\sigma_0(\omega),  \ldots, \sigma_n(\omega), \sigma'_0(\omega)  \ldots, \sigma'_n(\omega)$ \big(resp. $G^0,  \ldots, G^n, G'^0,$\\$  \ldots, G'^n$\big). Let  $(\bar{\delta_n})_{n\in \mathbb{N}}:=(\bar{G}^n, \bar{\sigma}_n)_{n\in \mathbb{N}}$. 
 It is easy to check that for each $n\in\mathbb{N}$, $\bar{\delta_n} \downarrow \delta$ in the sense of Definition \ref{RL}. Again, by the supermartingales convergence theorem, there exists a random variable $\bar{Z}$ such that $(U(\bar{\delta_n}))_{n\in \mathbb{N}}$ converges a.s. to $\bar{Z}$. Therefore, $$Z=\bar{Z}\quad \text{a.s.}$$ Indeed, for almost each $\omega\in \Omega$, the sequence $(\bar{\delta_n}(\omega))$ outlines all the values taken by the sequences $(\delta_n(\omega))$ and $(\delta'_n(\omega))$. In other words, for each $p$, there exists $n(\omega) \geq p$ such that $\bar{\delta}_n(\omega)= \delta_p(\omega)$. 
Since $U$ is an admissible family, for each $n,p \geq 0$, we have 
\begin{equation}
U(\delta_p)= U(\bar{\delta}_n )\,\, \text{a.s.} \quad \text{on} \quad \lbrace \bar{\delta}_n=\delta_p\rbrace.
\end{equation}
Let $\epsilon > 0$ and for almost each $\omega$, suppose that $Z(\omega)$ and $\bar{Z}(\omega)$ are finite. By definition of the limit, there exists $N(\omega)\geq 0$, such that for every $n,p \geq N(\omega)$, 
\begin{equation} \label{Eq3}
\vert U(\delta_p(\omega)) - Z(\omega) \vert \leq \epsilon \quad \text{and} \quad \vert U(\bar{\delta}_n(\omega)) - \bar{Z}(\omega) \vert \leq \epsilon. 
\end{equation}
Moreover, there exists $n_0(\omega) \geq N(\omega)$, such that $\delta_{N(\omega)}(\omega)= \bar{\delta}_{n_0(\omega)}(\omega)$. Thus, by Eq. (\ref{Eq3}), $$U(\delta_{N(\omega)}(\omega))= U(\bar{\delta}_{n_0(\omega)}(\omega)).$$ 
We conclude that $\vert Z(\omega) - \bar{Z}(\omega) \vert \leq 2\epsilon$ for each $\epsilon>0$, thus we have $Z(\omega)= \bar{Z}(\omega)$. Likewise, we prove that if $Z(\omega)=\infty$ or $\bar{Z}(\omega)= \infty$, then both are not finite. 
This gives $Z=\bar{Z}$ a.s. By symmetry, $Z'=\bar{Z}$ a.s., hence $Z=Z'$ a.s. The proof of the theorem is thus complete. \\
   \end{proof}
   
   Let us now give the definition of a right continuous admissible family along split stopping times, and some useful properties. 
   
   \begin{definition}(\textbf{Right continuous family along s.s.t.})
	An admissible family $(U(\delta),  \delta \in \mathcal{S}_0)$ is said to be right continuous along split
stopping times (RC in short) if for any $\delta \in \mathcal{S}_0$ and for any sequence of split stopping times $(\delta_n)_{n\in \mathbb{N}}$, such that $\delta_n \downarrow \delta$ one has $U(\delta) = \lim\limits_{n \rightarrow \infty} U(\delta_n)$ a.s.
  \end{definition}
  
  \begin{proposition} \label{v,v+ prop}
  Let $(U(\delta), \delta\in \mathcal{S}_0)$ be a uniformly integrable admissible supermartingale family. Then, the family $U$ has the following properties 
\begin{enumerate}
\item For each $\delta \in \mathcal{S}_0$, we have $U(\delta^+) \leq U(\delta)$ a.s. 
\item If $(U(\delta), \delta\in \mathcal{S}_0)$ is non-negative and RCE along split stopping times. Then,
 $(U(\delta), \delta$\\$\in \mathcal{S}_0) $ is RC, that is, for each $\delta \in \mathcal{S}_0$, $U(\delta^+)=U(\delta)$ a.s. 
\end{enumerate}  
  \end{proposition}
  \begin{proof}
  Let  $(\delta_n)_{n\in \mathbb{N}} =(G^n,\sigma_n)_{n\in \mathbb{N}}$ such that $\delta_n \downarrow \delta$, i.e. $G^n \downarrow G$, $\sigma_n \downarrow \sigma$ and for each $n$, $\delta_n >\delta$. \\
   \hspace*{0.5 cm} The supermartingale property of $U$ implies that $E[ U(\delta_n) \vert \mathcal{F}_\delta] \leq U(\delta)$ a.s. for every $n$. 
 Passing to the limit and using the uniform integrability of $(U(\delta_n), n\geq 0)$ yields $E[ \lim\limits_{n \rightarrow \infty} U(\delta_n) \vert \mathcal{F}_\delta] \leq U(\delta)$ a.s.
 On the other hand, from Theorem \ref{RL theorem}, we have  $U(\delta^+) = \lim\limits_{n \rightarrow \infty} U(\delta_n)$ a.s. Thus, $E[ U(\delta^+) \vert \mathcal{F}_\delta] \leq U(\delta)$ a.s. 
 Considering that $U(\delta^+)$ is $\mathcal{F}_{\delta}$-measurable, we get $U(\delta^+) \leq U(\delta)$ a.s. The first assertion thus holds. \\
 \hspace*{0.5 cm} We now prove the second assertion. The uniform integrability property of $(U(\delta_n), n\geq 0)$  together with the RL property of $U$ yields $$E[U(\delta^+)]= E[\lim\limits_{n \rightarrow \infty} U(\delta_n)] = \lim\limits_{n \rightarrow \infty} E[U(\delta_n)] \quad \text{a.s.}$$ From the RCE property of $U$ over split stopping times, we have $E[U(\delta^+)]= E[U(\delta)]$ a.s. By this and the first assertion $U(\delta^+) \leq U(\delta)$ a.s., we get the desired conclusion.\\ 
  \end{proof}
  
  We can now state the following property satisfied by the value functions families $v$ and $v^+$. 
  
 \begin{proposition}\label{v+= v(+)}
 Let $\xi$ be a nonnegative ladlag process in $\mathbf{S}^2$. Let $v$ and $v^+$ be the value functions families defined in Eq. (\ref{v}) and Eq. (\ref{v+}). For each $\delta \in \mathcal{S}_0$, we have 
 $$ v^+(\delta)= v(\delta^+) \quad \text{a.s.}$$
  \end{proposition}
 
 \begin{proof}
 Let $\delta \in \mathcal{S}_0$. We first observe that since $v^+$ is RCE over split stopping times, the second assertion of Proposition \ref{v,v+ prop} implies that $v^+$ is RC. Let  $(\delta_n)_{n\in \mathbb{N}} =(G^n,\sigma_n)_{n\in \mathbb{N}}$ such that $\delta_n \downarrow \delta$, i.e. $G^n \downarrow G$, $\sigma_n \downarrow \sigma$ and for each $n$, $\delta_n > \delta$. For each $n$, we have $v(\delta_n) \geq v^+(\delta_n)$ a.s. and $\lim\limits_{n \rightarrow \infty} v^+(\delta_n)= v^+(\delta)$ a.s. By letting $n$ tend to $\infty$, we get $v(\delta^+) \geq v^+(\delta)$ a.s. \\
It remains to show the other inequality. Since for each $n$, $\delta_n > \delta$, by Lemma \ref{lemma}, we have $v^+(\delta)\geq E[v(\delta_n) \vert \mathcal{F}_\delta]$ a.s. for every $n$. Passing to the limit and using the uniform integrability property of $(v(\delta_n), n\in \mathbb{N})$, we get $v^+(\delta)\geq E[v(\delta^+) \vert \mathcal{F}_\delta]= v(\delta^+)$ a.s., where the last equality follows from the $\mathcal{F}_{\delta}$-measurability of $v(\delta^+)$. Hence, $ v^+(\delta)= v(\delta^+)$ a.s.\\
 \end{proof}
 
 \begin{remark} \label{v+(sigma-)}
Since $\mathcal{T}_0 \subset \mathcal{S}_0$, we consider the family $\lbrace v^+(\sigma), \sigma \in \mathcal{T}_0 \rbrace$ where $ v^+(\sigma)$ is defined for each $\sigma \in \mathcal{T}_0$, by
$$v^+(\sigma):= v^+(\delta), \quad \delta=(\emptyset, \sigma)\in \mathcal{S}_0.$$
The family $v^+=(v^+(\delta), \delta\in\mathcal{S}_0)$ is a supermartingale family over split stopping times, hence one can see that the family $\lbrace v^+(\sigma), \sigma \in \mathcal{T}_0 \rbrace$ is a supermartingale family over stopping times (see \cite[p. 5]{KobylanskiQuenez}). Consequently, the family $\lbrace v^+(\sigma), \sigma \in \mathcal{T}_0 \rbrace$ is left limited along stopping times at each time $\sigma \in \mathcal{T}_{0^+}$ (e.g., \cite[Theorem 4.3]{KobylanskiQuenez}) i.e., there exists an $\mathcal{F}_{\sigma^-}$-measurable random variable $v^+(\sigma^-)$ such that, for any nondecreasing sequence of stopping times $(\sigma_n)_{n \in \mathbb{N}}$, 
\begin{equation}
v^+(\sigma^-)=  \lim\limits_{n \rightarrow +\infty} v^+(\sigma_n) \quad \text{a.s. on} \quad A[(\sigma_n)],
\end{equation} 
where $A[(\sigma_n )] = \lbrace \sigma_n \uparrow \sigma$ and $ \sigma_n < \sigma$ for all $n \rbrace$.
\end{remark}
  
Let $\rho^T:= (\emptyset, T)$. Let $\xi$ be a nonnegative ladlag process in $\mathbf{S}^2$. Using the above results on the value functions families $v=(v(\delta), \delta\in\mathcal{S}_0)$ and $v^+=(v^+(\delta), \delta\in\mathcal{S}_0)$, we show the following main result, which generalizes the classical Snell envelope of the pay-off process $\xi$ to the case of split stopping times. 

\begin{theorem}(\textbf{Aggregation and Snell envelope}) \label{snellTh}
There exists a positive optional process, denoted by $(v_t)_{t\in[0,T]}$, which aggregates the value function family $v=(v(\delta), \delta\in\mathcal{S}_0)$, that is, for each $\delta \in \mathcal{S}_0$, $v_\delta  = v(\delta)$ a.s. Moreover, $(v_t)_{t\in[0,T]}$ is the Snell envelope of the pay-off process $\xi$, that is, the smallest strong supermartingale greater than or equal to $\xi$. It is called the Snell envelope of $\xi$ over split stopping times.
\end{theorem}

\begin{proof}
Let us consider the process $(v^+(t))_{t\in [0,T]}$. Since the supermartingale family $(v^+(\delta), \delta\in\mathcal{S}_0)$ is RCE over split stopping times, the process $(v^+(t))$ is a supermartingale and the function $t \rightarrow E[v^+(t)]$ is right-continuous. From this, Eq. \eqref{Xrho} and Remark \ref{v+(sigma-)}, we have for each $\delta=(G,\sigma) \in \mathcal{S}_0$
$$v^+(\delta):= v^+(\sigma^-) \mathbf{1}_{G} +  v^+(\sigma) \mathbf{1}_{G^c}.$$  \\
From a classical result (see for instance Theorem 3.13 in Karatzas and Shreeve \cite{Kara} and Proposition 4.1 in Kobylanski et al. \cite{Kobylanski}), there exists a rcll supermartingale $(v^+_t)_{t\in[0,T]}$ such that for each $t \in [0,T]$, $v^+_t= v^+(t)$ a.s. It is easy to check that for each dyadic stopping time $\sigma \in \mathcal{T}_0$, we have $v^+_\sigma= v^+(\sigma)$ and 
\begin{equation} \label{Eq100} 
E[v^+_{\sigma}]=E[v^+(\sigma)].
\end{equation}
Let now $\sigma$ be a stopping time in $\mathcal{T}_0$, one can show that equality \eqref{Eq100} still holds for $\sigma$. Indeed, 
\begin{equation}
\sigma_n= \sum_{k=0}^\infty \frac{k+1}{2^n} \mathbf{ 1}_{\lbrace \frac{k}{2^n} <\sigma\leq \frac{k+1}{2^n} \rbrace} + T \mathbf{1}_{\lbrace\sigma=T\rbrace}, \quad n=0,1,2, \ldots
\end{equation}
is a decreasing sequence of stopping times such that $\sigma= \lim\limits_{n \rightarrow + \infty} \downarrow \sigma_n$. Since the process $(v^+_t)_{0 \leq t \leq T}$ is rcll and the family  $(v^+(\delta),\delta\in\mathcal{S}_0)$ is RCE over split stopping times, from dominated converegence theorem we get
\begin{equation}
E[v^+_\sigma]= \lim\limits_{n \rightarrow + \infty} E[v^{+}_{\sigma_n}]=  
\lim\limits_{n \rightarrow + \infty} E[v^+(\sigma_n)]= E[v^+(\sigma)]. 
\end{equation} 
Let $A\in \mathcal{F}_{\sigma}$ and define $\sigma_A= \sigma \mathbf{1}_A + T \mathbf{1}_{A^c}$,  $\sigma_A$ is a stopping time hence $E[v^{+}_{\sigma_A}]= E[v^+(\sigma_A)]$. Since $v^+_T= v^+(T)$ a.s., it follows that $E[v^+_{\sigma} \mathbf{1}_A]= E[v^+(\sigma) \mathbf{1}_A ]$, which implies that 
\begin{equation} \label{Eq10} 
v^+_\sigma= v^+(\sigma) \quad \text{a.s.}
\end{equation}
Let $(\sigma_n)_{n\in \mathbb{N}}$ be a nondecreasing sequence of stopping times such that $\sigma_n \uparrow \sigma$ and $\sigma_n < \sigma$ for all $n$. From \eqref{Eq10}, we have $v^+_{\sigma_n}= v^+(\sigma_n)$, since $(v^+_t)$ is left limited and the family  $\lbrace v^+(\sigma), \sigma \in \mathcal{T}_0 \rbrace$ is left limited along stopping times, Remark \ref{v+(sigma-)} implies that 
\begin{equation} \label{Eq11}
v^+_{\sigma^-}= \lim\limits_{n \rightarrow +\infty} v^+_{\sigma_n} = \lim\limits_{n \rightarrow +\infty} v^+(\sigma_n)= v^+(\sigma^-) \quad \text{a.s.}
\end{equation}
On the other hand, for each $\delta=(G,\sigma) \in \mathcal{S}_0$, $v^+_\delta$ is defined by: $v^+_\delta:= v^+_{\sigma^-} \mathbf{1}_G + v^+_\sigma \mathbf{1}_{G^c}$. From \eqref{Eq10} and \eqref{Eq11}, we get 
\begin{equation} \label{Eq4}
v^+_\delta= v^+(\sigma^-) \mathbf{1}_G + v^+(\sigma )\mathbf{1}_{G^c} = v^+(\delta).
\end{equation}
The process $(v^+_t)$ thus aggregates the family $\lbrace v^+(\delta), \delta \in \mathcal{S}_0 \rbrace$. Let now the process $(v_t)_{t\in[0,T]}$ defined by 
\begin{equation}
v_t := \xi_t \vee v^+_t.
\end{equation} 
By Proposition \ref{v=xi vee v+} and the above, we clearly have $v(\delta)=v_\delta$, for each $\delta\in\mathcal{S}_0$. From Proposition \ref{v v+ admiss} and Proposition \ref{snell env family}, it follows that $(v_t)_{t\in[0,T]}$ is the smallest strong supermartingale bounding $\xi$ above (except on evanescent sets), which makes the proof ended.\\
\end{proof}

Note that since the process $(v_t)_{t\in[0,T]}$ is a strong supermartingale, it follows from Theorem 4 in \cite[Appendix 1, p. 408]{Probabilites et Potentiel2} that $(v_t)$ has right and left limits. We thus derive the following Theorem:

\begin{theorem} \label{thmpro}
\begin{enumerate}
\item[(i)] The strict value process $(v^+_t)_{t\in[0,T]}$ is right-continuous.
\item[(ii)] For each  $\delta \in \mathcal{S}_0$, $v^+_\delta=v_{\delta^+}$ a.s.
\item[(iii)] For each $\delta \in \mathcal{S}_0$, $v_\delta= v_{\delta^+}  \vee \xi_\delta$ a.s.
\item[(iv)] For each $\theta \in \mathcal{T}_0$ and for each  $\lambda \in ]0, 1[$, the process $(v_t)_{t\in[0,T]}$ is a martingale on $[\theta, \tau^{\lambda}_{\theta}]$, where $ \tau^\lambda_\theta:= \inf \lbrace t \geq \theta, \lambda v_t (\omega) \leq \xi_t \rbrace$.  
\end{enumerate}
\end{theorem}

\begin{proof} Statement (i) is a direct consequence of Proposition \ref{v+RCE}, together
with part 2. of Proposition \ref{v,v+ prop}. Let us prove the second statement. From Proposition \ref{v+= v(+)} and Eq. \ref{Eq4}, we have for each $\delta \in\mathcal{S}_0$, $ v^+(\delta)=v^+_\delta= v(\delta^+)$ a.s. On the other hand, the process $(v_t)$ aggregates the family $(v(\delta), \delta \in \mathcal{S}_0)$, which implies that the process $(v_{t^+})$ aggregates the family $(v^+(\delta), \delta \in \mathcal{S}_0)$, that is, for each $\delta \in \mathcal{S}_0$, $v(\delta^+)=v_{\delta^+}$ a.s. Therefore, we get $ v^+(\delta)= v^+_\delta=v_{\delta^+}$ a.s., which is the expected result.
Assertion (iii) follows immediately from part (ii) (which we have just shown), together with Proposition \ref{v=xi vee v+}.\\ The last assertion corresponds to a known result of optimal stopping theory and can be shown in our case, using a penalization method, introduced by Maingueneau (cf. the proof of Theorem 2 in \cite{Maingueneau}). 
\end{proof}

\begin{remark}\label{rmkk}
On account of (iii), we have for each $\delta \in \mathcal{S}_0$, $\Delta_+ v_\delta = \mathbf{1}_{\lbrace v_\delta = \xi_\delta \rbrace} \Delta_+ v_\delta $ a.s.
\end{remark}

The following Lemma holds:

\begin{lemma} \label{mertens}
\begin{enumerate}
\item[(i)] The value process $v$ belongs to $\mathcal{S}^2$, and there exists a unique cadlag uniformly integrable martingale $M \in\mathbf{M}^2$, a unique nondecreasing right-continuous predictable process $A$ with
$A_0 = 0$, $E(A_{T}^2 ) <  \infty$, and a unique nondecreasing right-continuous adapted purely
discontinuous process $B$ with $B_{0^-} = 0$, $E(B_{T}^2) < \infty$, such that
\begin{equation}\label{Eq8}
v_\delta= v_0 + M_\delta- A_\delta - B_{\delta^-} \quad \text{a.s.} \quad \text{for all} \,\, \delta\in\mathcal{S}_0.
\end{equation}
\item[(ii)] For each $\delta \in \mathcal{S}_0$, we have $\Delta B_\delta = \mathbf{1_{\lbrace v_\delta =\xi_\delta \rbrace}} \Delta B_\delta$ a.s.
\item[(iii)] For each $\delta \in \mathcal{S}^p_0$, we have $\Delta A_\delta = \mathbf{1_{\lbrace v_{\delta^-} =\xi_{\delta^-} \rbrace}} \Delta A_\delta$ a.s.
\item[(iv)] The continuous part $A^c$ of $A$ satisfies the equality $\int_0^{\rho^T} \mathbf{1}_{\lbrace v_{t^-} > \xi_{t^-}\rbrace} dA^c_t =0$ a.s. 
\end{enumerate}
\end{lemma}
\begin{proof}
Let us prove the first point. By Theorem \ref{snellTh}, the value process $(v_t )_{t\in[0,T ]}$ is a strong supermartingale. Moreover, let $ \sigma \in \mathcal{T}_0$, from Remark \ref{rmk}, we have $\delta:=(\emptyset, \sigma)\in\mathcal{S}_0$, $v_\delta=v_\sigma$ and $\mathcal{F}_\delta\equiv \mathcal{F}_\sigma$. Thus, using Jensen's inequality yields 
\begin{align} \label{Eq6} 
\vert v_\sigma\vert \leq \esssup_{\rho \in\mathcal{S}_{\delta}} E[\vert \xi_\rho\vert \vert \mathcal{F}_\sigma] \leq E[\esssup_{\rho \in\mathcal{S}_{\delta}} \vert \xi_\rho\vert \vert \mathcal{F}_\sigma]= E[X \vert \mathcal{F}_\sigma] \quad \text{a.s. for all $\sigma \in \mathcal{T}_0$},
\end{align}
where $X$ is $\esssup\limits_{\rho \in\mathcal{S}_{\delta}} \vert \xi_\rho\vert$. 
By using Eq. (\ref{Eq5}), we get 
\begin{align}
E[X^2]= E\Big[\esssup\limits_{\rho \in\mathcal{S}_{\delta}} \vert \xi_\rho \vert^2\Big]&=  E\Big[\esssup\limits_{\substack{\rho \in\mathcal{S}_{\delta} \nonumber \\ \rho:=(H,\tau)}}\, (\vert \xi_{\tau^-} \vert^2 \mathbf{1}_{H} + \vert \xi_{\tau} \vert^2 \mathbf{1}_{H^c}) \Big] \\
&\leq E\Big[\esssup\limits_{\substack{\rho \in\mathcal{S}_{\delta} \nonumber \\ \rho:=(H,\tau)}} \, (\esssup\limits_{\theta \in \mathcal{T}_0} \vert \xi_{\theta^-} \vert^2  \mathbf{1}_{H} + \esssup\limits_{\theta \in \mathcal{T}_0} \vert \xi_{\theta} \vert^2 \mathbf{1}_{H^c} ) \Big] \\ \leq  \Vert \xi \Vert^2_{\mathcal{S}^2}. \label{Eq7}
\end{align}
From Eq. (\ref{Eq6}), we have 
\begin{equation*}
    E[\esssup\limits_{\sigma \in\mathcal{T}_0}\vert v_\sigma\vert^2]\leq E[\esssup\limits_{\sigma \in\mathcal{T}_0} \vert E[X \vert \mathcal{F}_\sigma]\vert^2 ]= E[\esssup\limits_{t\in [0,T]} \vert E[X \vert \mathcal{F}_t]\vert^2 ],
\end{equation*}
where the last equality follows from the right-continuity of the martingale $(E[X\vert \mathcal{F}_t])_{t \in [0,T]}$.
Using Doob's martingale inequalities in $L^2$, we obtain 
\begin{equation*}
    E[\esssup\limits_{\sigma \in\mathcal{T}_0} \vert v_\sigma \vert^2] \leq 4 E[X^2] \leq 4\Vert \xi \Vert^2_{\mathcal{S}^2}< \infty.
\end{equation*}
Therefore, the process $(v_t )_{t\in[0,T ]}$ is a strong supermartingale in $\mathcal{S}^2$ (thus, of class $(D)$), by Mertens decomposition for
strong supermartingales of class $(D)$ (see for example equalities (20.2) of Theorem 20 in \cite{Probabilites et Potentiel2}), there exists a unique uniformly integrable cadlag martingale $(M_t)$, a unique predictable right continuous nondecreasing process $(A_t)$ with $A_0 = 0$ and $E[A_T] < \infty$ and a unique right continuous adapted nondecreasing process $(B_t)$, which is purely discontinuous with $B_{0^-} = 0$ and $E[B_T] < \infty$, such that
 \begin{equation}
v_t= v_0 + M_t- A_t - B_{t^-}, \quad \text{for all} \,\, t\in[0,T]\quad \text{a.s.}, 
\end{equation}
which gives for each $\delta\in\mathcal{S}_0$, $v_\delta= v_0 + M_\delta- A_\delta - B_{\delta^-}$ a.s. What is left is to show that $A_T$, $B_T$ $\in L^2$. By Eq. (\ref{Eq6}), we have for every $\sigma\in\mathcal{T}_0$, $\vert v_\sigma \vert \leq E[X\vert \mathcal{F}_\sigma]$ a.s., where $X=\esssup\limits_{\rho \in\mathcal{S}_{\delta}} \vert \xi_\rho\vert$ is a non-negative $\mathcal{F}_T$-measurable random variable. From this, Eq. (\ref{Eq7}) and a fine result of Dellacherie-Meyer (cf., e.g., \cite[Corollary A.1]{Reflected BSDEs when the obstacle is not right-continuous and optimal stopping}), yields $E(A_T + B_{T^-}) \leq cE[X^2]\leq c \Vert \xi \Vert^2_{\mathcal{S}^2}$, where $c > 0$. We conclude that $A_T$ and $B_T(= B_{T^-})$ are square integrable, which gives the statement (i).\\ By Eq. (\ref{Eq8}), we have for each $\delta \in \mathcal{S}_0$, $\Delta B_\delta = - \Delta_+ v_\delta$ a.s.  As from  Remark \ref{rmkk}, $\Delta_+ v_\delta = \mathbf{1}_{\lbrace v_\delta = \xi_\delta \rbrace} \Delta_+ v_\delta $ a.s., the second point clearly holds.\\
Let us show assertion (iii). Observe that, for each $\delta=(G, \sigma) \in\mathcal{S}^p_0$, $\Delta A_\delta= \Delta A_\sigma \mathbf{1}_{G^c}$, hence, it suffices to show that for each $\theta\in\mathcal{T}^p_0$,  \begin{equation} \label{Eq9}
\Delta A_\theta= \mathbf{1}_{ \lbrace v_{\theta^-}= \xi_{\theta^-} \rbrace} \Delta A_\theta. 
\end{equation}
 Theorem \ref{thmpro} (iv) together with Mertens decomposition (\ref{Eq8}), gives $A_\theta=A_{\tau^\lambda_\theta}$ a.s. for each $\theta\in\mathcal{T}_0$, and for each $\lambda\in ]0,1[$. From this and a result of El Karoui (\cite[Proposition 2.34]{El Karoui}), Eq. (\ref{Eq9}) follows, which yields the third assertion. \\
  Based on Theorem \ref{thmpro} (iv) and on some analytic arguments, the proof of (iv) follows by the same method as in the proof of Lemma 3.3 \cite{Optimal stopping with}. 
\end{proof}
\begin{remark} 
\begin{enumerate}
\item In the case of a not necessarily non-negative reward process $\xi \in \mathcal{S}^2$, we define $X_\delta:= E[ \esssup\limits_{\rho=(H,\tau) \in \mathcal{S}_0}  \xi^{-}_\rho \vert \mathcal{F}_\delta]$ and $\bar{\xi}_\delta:= \xi_\delta + X_\delta$. The reward process $\bar{\xi}$ is non-negative and is associated with the value family $\bar{v}:= v+ X$. The results of this section thus hold for $\xi$ up to a translation by the martingale $X$ (cf. e.g., \cite[Appendix A]{KobylanskiQuenezRoger}).
\item Section \ref{sect3.1} can be extended to the setup of reward families indexed by split stopping times, i.e. $(\xi(\delta), \delta \in\mathcal{S}_0)$, such a generalization has been studied by Kobylanski and Quenez  \cite{KobylanskiQuenez} in the framework of stopping times. %Therefore, we can abandon the condition of existence of left and right limits on the reward process $\xi$ and the results still be valid.  
\end{enumerate} 
\end{remark} 
\subsection{Existence and uniqueness of the solution of Reflected BSDEs with a ladlag obstacle} \label{sub 3.2}

\hspace*{0.5 cm} Let $\rho^T=(\emptyset, T)$. In this section, we prove existence and uniqueness of the solution to the RBSDE from Definition \ref{def2} (in the case where the driver $g$ does not depend on $y$, $z$). We characterize the first component of the solution as the value process of an optimal stopping problem over split stopping times. To this end, note first that we have the following a priori estimates.

\begin{lemma} \label{a prio}
 Let $(Y,Z,M,A,B)$ (resp. $(\bar{Y}, \bar{Z}, \bar{M}, \bar{A}, \bar{B})) \in \mathbf{S}^{2} \times \mathbf{H}^2 \times \mathbf{M}^{2,\perp} \times (\mathbf{S}^{2}) \times (\mathbf{S}^{2})$ be a solution to the RBSDE associated with driver $g=(g_t) \in \mathbf{H}^{2}$ (resp. $\bar{g}=(\bar{g}_t) \in \mathbf{H}^{2}$) and with a ladlag obstacle $\xi$. Then, there exists $c>0$ such that for all $\epsilon \geq 0$, for all $\beta >\frac{1}{\epsilon^2}$, we have 
\begin{equation}\label{l1}
\Vert Z- \bar{Z} \Vert^2_{\beta} + \Vert M- \bar{M} \Vert^2_{\mathbf{M}^2_\beta} \leq \epsilon^2 \Vert g- \bar{g} \Vert^2_{\beta}; \\
\end{equation}
\begin{equation} \label{l11}
\vertiii{Y- \bar{Y}}^2_{\beta} \leq 4\epsilon^2(1 + 12c^2) \Vert g- \bar{g} \Vert^2_{\beta}.
\end{equation} 
\end{lemma}

\begin{proof} 
The proof relies on Gal'chouk-Lenglart's formula (cf. Theorem \cite{Gal'chouk}) and employs the same reasoning as that used in the proof of Lemma 3.7 in \cite{Optimal stopping with}.
\end{proof}

We assume that $g$ does not depend on $(y,z)$ i.e., $P$-a.s., $g(t, \omega, y, z) \equiv g (t, \omega)$, for any $t, y$ and $z$. Using these a priori estimates, the lemmas from the previous subsection, and the orthogonal martingale decomposition (see Lemma \ref{orthogonal decomposition}), we prove existence and uniqueness of the solution to the
RBSDE from Definition \ref{def2}. 

\begin{theorem} \label{existence RBSDE}
Let $\xi$ be a ladlag process in  $\mathbf{S}^{2}$ and $g=(g_t)$ a driver process in $\mathbf{H}^2$. Then,
the RBSDE from Definition \ref{def2} admits a unique solution
$(Y,Z,M,A,B) \in  \mathbf{S}^{2} \times \mathbf{H}^2 \times \mathbf{M}^{2,\perp} \times (\mathbf{S}^{2})^2$, and for each 
$\delta=(G,\sigma) \in \mathcal{S}_0$, we have 
\begin{equation}
Y_\delta= \esssup_{\rho=(H,\tau) \in \mathcal{S}_{\delta}} E[ \xi_\rho + \int_\delta^\rho g(t) dt \vert \mathcal{F}_\delta]\quad \text{a.s.}
\end{equation}
Moreover, the following property holds 
\begin{equation}
Y_\delta = \xi_\delta \vee Y_{\delta^+} \,\,\, \text{a.s.}
\end{equation}
\end{theorem}

\begin{proof} For each $\delta=(G,\sigma) \in \mathcal{S}_0$, we define the value $Y(\delta)$ at the split stopping time $\delta$ by
\begin{equation}
Y(\delta) := \esssup_{\rho=(H,\tau) \in \mathcal{S}_{\delta}} E[ \xi_\rho + \int_\delta^\rho g(t) dt \vert \mathcal{F}_\delta].
\end{equation}
This is equivalent to 
\begin{equation}
Y(\delta) + \int_0^\delta g(t) dt := \esssup_{\rho=(H,\tau) \in \mathcal{S}_{\delta}} E[ \xi_\rho + \int_0^\rho g(t) dt \vert \mathcal{F}_\delta].
\end{equation}
Thus, Theorem \ref{snellTh} holds with $\xi$ replaced by $\xi_. + \int_0^. g(t) dt$ and $v(\delta)$ by $Y(\delta)+ \int_0^\delta g(t) dt$. Then, there exists a ladlag optional process $(Y_t)_{t \in [0,T]}$ which aggregates the family\\ $(Y(\delta))_{\delta \in \mathcal{S}_0}$, that is,
\begin{equation}
Y_\delta = Y(\delta) \,\, \text{a.s for all}\,\,\, \delta \in \mathcal{S}_0.
\end{equation}
Hence, $Y_\delta \geq \xi_\delta$ a.s. for all $\delta \in \mathcal{S}_0$. For each $\rho \in \mathcal{S}_{\rho^T}$, we have $\rho=\rho^T$, consequently $Y_{\rho^T} = Y(\rho^T) = \xi_{\rho^T}$.
Moreover, the process $(Y_t + \int_0^t g(t) dt)_{t\in [0,T]}$ is the smallest strong supermartingale greater than or equal to $(\xi_t + \int_0^t g(t) dt)_{t\in [0,T]}$, and from (iii) of Theorem \ref{thmpro}, we have 
$$ Y_\delta = \xi_\delta \vee Y_{\delta^+} \,\,\, \text{a.s.}$$
Next, from Lemma \ref{mertens}, $Y$ belongs to $\mathbf{S}^2$ and admits the following Mertens decomposition:
\begin{equation} \label{Yoo}
Y_\delta= Y_0 - \int_0^\delta g(s) ds + M_\delta - A_\delta -B_{\delta^-}, \quad \text{a.s.} \quad \text{for all} \,\, \delta\in\mathcal{S}_0.
\end{equation}
where $M\in\mathbf{M}^2$, $A$ is a nondecreasing right-continuous predictable process such that
$A_0 = 0$, $E(A_{\rho^T}^2)= E(A_T^2 ) <  \infty$, and $B$ is a nondecreasing right-continuous adapted purely
discontinuous process such that $B_{0^-} = 0$, $E(B_{\rho^T}^2)=E(B_T^2) < \infty$.
Furthermore,
\begin{enumerate}
\item[(i)] For each $\delta \in \mathcal{S}_0$, we have $\Delta B_\delta = \mathbf{1_{\lbrace Y_\delta =\xi_\delta \rbrace}} \Delta B_\delta$ a.s.,
\item[(ii)] For each $\delta \in \mathcal{S}^p_0$, we have $\Delta A_\delta = \mathbf{1_{\lbrace Y_{\delta^-} =\xi_{\delta^-} \rbrace}} \Delta A_\delta$ a.s.,
\item[(iii)] The continuous part $A^c$ of $A$ satisfies the equality $\int_0^{\rho^T} \mathbf{1}_{\lbrace Y_{t^-} > \xi_{t^-}\rbrace} dA^c_t =0$ a.s. 
\end{enumerate}
 From (i), the process $B$ satisfies the minimality condition \ref{sko2'}, and from (ii) and (iii), the process $A$ satisfies condition \ref{sko1'}. 
Finally, from Lemma \ref{orthogonal decomposition}, there exists a unique couple $(Z,N)\in \mathbf{H}^2 \times \mathbf{M}^{2,\bot} $ such that
\begin{equation}
M_t= \int_0^t Z_s dW_s + N_t, \,\,\,\,  0\leq t\leq T \,\,\,\, a.s.
\end{equation}
From this together with (\ref{Yoo}) we get, for each $\delta \in \mathcal{S}_0$
\begin{align*}
Y_\delta=\xi_{\rho^T} + \int_\delta^{\rho^T}  g(s) ds &- \int_\delta^{\rho^T} Z_s dW_s - (N_{\rho^T}-N_{\delta}) +  A_{\rho^T}-A_\delta \\
&+ B_{{\rho^T}^-}-B_{\delta^-} \,\,\, \text{a.s}. 
\end{align*}
The process $(Y,Z,N,A,B)$ is thus a solution of the RBSDE (\ref{def2}) associated with the driver process $g(t)$ and the obstacle $\xi$. The uniqueness of $A$, $B$, $Z$ and $N$ follows from the uniqueness of Mertens decomposition and from
the uniqueness of the martingale representation (Lemma \ref{orthogonal decomposition}). If $(\bar{Y},Z,N,A,B)$ is a solution of the RBSDE with driver $g$ and obstacle $\xi$, by inequality (\ref{l11}) of Lemma \ref{a prio}, we obtain $\bar{Y}= Y$ in $\mathcal{S}^2$. The proof of the Theorem is thus complete. \\
\end{proof}

Let now $g$ be a general Lipschitz driver. Using Theorem \ref{existence RBSDE} together with the a priori estimates from Lemma \ref{a prio}, we derive the following existence and uniqueness result in the case of a general Lipschitz driver $g(t,y,z)$.

\begin{theorem} \label{uniqueness}
Let $\xi$ be a ladlag process in $\mathbf{S}^2$ and let $g$ be a Lipschitz driver in $\mathbf{H}^2$. The RBSDE associated with the parameters $(g,\xi)$ from Definition \ref{def2} admits a unique solution $(Y,Z,M,A,B) \in  \mathbf{S}^{2} \times \mathbf{H}^2 \times \mathbf{M}^{2,\perp} \times (\mathbf{S}^{2})^2$.
\end{theorem}

\begin{proof}
We introduce a mapping $\Phi$ from $\mathbf{K}^2_\beta$ into itself. This map is defined by: for a given $(U,V)\in \mathbf{K}^2_\beta$, $\Phi(U,V) := (Y,Z)$, where $Y$, $Z$ are the first two components of the solution $(Y,Z,M,A,B)$ to the reflected BSDE associated with driver $g_t:= g(t,U_t,V_t)$ and with the obstacle $\xi$.
 By Theorem \ref{existence RBSDE}, the mapping $\Phi$ is well-defined. \\
\vspace*{0.5 cm}  Consider $(U,V)$ and $(\bar{U},\bar{V})$ two elements of $\mathbf{K}^2_\beta$. We set $(Y,Z)= \Phi(U,V)$, $(\bar{Y}, \bar{Z})= \Phi(\bar{U},\bar{V})$, $\tilde{Y}:= Y - \bar{Y}$, $\tilde{Z}:= Z - \bar{Z}$, $\tilde{U}:= U - \bar{U}$ and $\tilde{V}:= V - \bar{V}$. Hence, Lemma \ref{a prio} shows that 
 \begin{equation}
 \vertiii{\tilde{Y}}^2_\beta + \Vert \tilde{Z} \Vert^2_{\beta} \leq 6\epsilon^2(1 + 8c^2) \Vert g(t, U_t,V_t)- g(t,\bar{U}_t, \bar{V}_t) \Vert^2_{\beta},
 \end{equation}
 for all $\epsilon>0$ and $\beta \geq \frac{1}{\epsilon^2}$. 
 By using the Lipschitz property of $g$, we obtain
 \begin{align}\label{m1}
  \vertiii{\tilde{Y}}^2_\beta + \Vert \tilde{Z} \Vert^2_{\beta}  \leq 6\epsilon^2 M (1+ 8 c^2) \Big[\Vert \tilde{U} \Vert^2_{\beta} + \Vert \tilde{V} \Vert^2_{\beta} \Big],
 \end{align}
 where $M$ is a positive constant depending only on the Lipschitz constant $K$. 
Using Fubini's theorem, we get $\Vert \tilde{U} \Vert^2_{\beta} \leq T \vertiii{\tilde{U}}^2_\beta$. 
This, combined with  (\ref{m1}), gives 
\begin{align}
  \vertiii{\tilde{Y}}^2_\beta + \Vert \tilde{Z} \Vert^2_{\beta}  \leq 6\epsilon^2 M(1+T)(1+ 8 c^2)  \Big[ \vertiii{\tilde{U}}^2_\beta + \Vert \tilde{V} \Vert^2_{\beta} \Big].
 \end{align}
Consequently, by choosing $\epsilon>0$ such that $6\epsilon^2 M(1+T)(1+ 8 c^2)<1$ and $\beta$ such that $\beta \geq \frac{1}{\epsilon^2}$, we deduce that the mapping $\Phi$ is a contraction, and hence by the Banach fixed-point theorem, $\Phi$ admits a unique fixed point $(Y,Z) \in \mathbf{K}_\beta^2$. The process $(Y,Z)$ is equal to the first two components of the solution $(Y,Z,M,A,B)$ to the RBSDE associated with the driver process $h(\omega,t):= g(\omega,t,Y_t(\omega), Z_t(\omega))$ and with the barrier $\xi$. Thus, $(Y,Z,M,A,B)$ is the unique solution to the RBSDE  with parameters $(g,\xi)$, which completes the proof. \\
\end{proof}

We introduce the following operator: 

\begin{definition}(\textbf{Operator induced by a RBSDE with driver 0})\\ \label{operator}
Let $\xi$ be a ladlag process in  $\mathbf{S}^{2}$. We denote by $\mathcal{R}ef[\xi]$ the first component of the solution to the reflected BSDE from Definition \ref{def2} in the case where the driver is $0$. Note that by Proposition \ref{existence RBSDE}, the operator $\mathcal{R}ef : \xi \rightarrow \mathcal{R}ef[\xi]$ is well defined on $\mathbf{S}^{2}$.
\end{definition}

 Here are some elementary properties of this operator.
 
\begin{lemma} \label{incr opera}
The operator $\mathcal{R}ef$ is nondecreasing, i.e. for $\xi, \xi' \in \mathbf{S}^{2}$ such that $\xi \leq \xi'$ we have $\mathcal{R}ef[\xi] \leq \mathcal{R}ef[\xi']$. Further, for each $\xi \in \mathbf{S}^{2}$, $\mathcal{R}ef[\xi]$ is a strong supermartingale and satisfies $\mathcal{R}ef[\xi] \geq \xi$. If moreover $\xi$ is a strong supermartingale, then $\mathcal{R}ef[\xi]=\xi$. 
\end{lemma}

\begin{proof} 
By definition, $\mathcal{R}ef[\xi]$ is the first component of the solution to the reflected BSDE (\ref{RBSDE}). Hence, Theorem \ref{existence RBSDE} shows that $\mathcal{R}ef[\xi]$ is the value function associated with the reward $\xi$, that is for each split stopping time $\delta=(G,\sigma) \in \mathcal{S}_0$ 
\begin{equation*}
\mathcal{R}ef[\xi]_\delta = \esssup_{\rho \in \mathcal{S}_\delta} E(\xi_\rho \vert \mathcal{F}_{\delta}).
\end{equation*}
Thus, the operator $\mathcal{R}ef$ is nondecreasing and the process $(\mathcal{R}ef[\xi])_{t\in [0,T]}$  is characterized as the Snell envelope over split stoping times associated with the process $(\xi)_{t\in [0,T]}$.
If now $\xi$ is a strong supermartingale, it remains to show that
$\mathcal{R}ef[\xi] \leq \xi$. Let $\delta \in \mathcal{S}_0$, since $\xi$ is a strong supermartingale, for each split stopping time $\rho \in \mathcal{S}_\delta $, we have
$$ E(\xi_\rho \vert \mathcal{F}_{\delta}) \leq \xi_\delta.$$
By definition of the essential supremum, we get $\mathcal{R}ef[\xi]_\delta \leq \xi_\delta$. Consequently, $\mathcal{R}ef[\xi] = \xi$.
\end{proof}

\begin{remark} \label{cv seq of p.s.s}
Observe that the nondecreasing limit of a sequence of strong supermartingales is a strong supermartingale (which is easy to check using Lebesgue theorem and the fact that the trajectories of strong supermartingales are bounded on compact intervals of $\mathbb{R}^+$). 
\end{remark}

\section{Doubly RBSDEs whose obstacles are irregular over a larger set of “stopping strategies”} \label{s4}

\hspace*{0.5 cm} We associate with $T$ the split terminal time $\rho^T=(\emptyset, T)$. Let now $g$ be a driver and $(\xi, \zeta)$ be a pair of ladlag admissible obstacles.

\begin{definition} \label{def1}
A process $(Y,Z,M,A,B,A',B') \in \mathbf{S}^{2} \times \mathbf{H}^2 \times \mathbf{M}^{2,\perp} \times (\mathbf{S}^{2})^2 \times (\mathbf{S}^{2})^2 $ is said to be a solution to the doubly RBSDE with parameters $(g,\xi,\zeta)$, where $g$ is a driver and $(\xi, \zeta)$ is a pair of admissible obstacles, if for all $\delta \in \mathcal{S}_0$
%\begin{equation}
\begin{multline}
Y_\delta=\xi_{\rho^T} + \int_\delta^{\rho^T} g(s,Y_s,Z_s) ds - \int_\delta^{\rho^T} Z_s dW_s - (M_{{\rho^T}}-M_{\delta}) +  A_{\rho^T} -A_\delta \\ - (A'_{\rho^T} -A'_\delta) + B_{{\rho^T}^-}-B_{\delta^-}-(B'_{{\rho^T}^-}-B'_{\delta^-}), \,\,\,\, \text{for all} \,\,\, \text{a.s.}, \label{eq*}
\end{multline}
%\end{equation}
with 
\begin{enumerate}
\item[(i)] $ \xi_\delta \leq Y_\delta \leq \zeta_\delta$,  a.s. for all  $\delta \in \mathcal{S}_0$,
\item[(ii)] $A$ and $A'$ are nondecreasing right-continuous predictable processes with $A_0=A'_0=0$ and such that 
\begin{align}
&\int_0^{\rho^T} \textbf{1}_{\lbrace Y_{t^-}>\xi_{t^-} \rbrace} dA^c_t=0 \,\,\, \text{a.s.    and    } \int_0^{\rho^T} \textbf{1}_{\lbrace Y_{t^-}<\zeta_{t^{-}} \rbrace} dA'^c_t =0  \,\,\, \text{a.s.}  \label{sko1} \\
& (Y_{\delta^-} - \xi_{\delta^-})(A^d_\delta - A^d_{\delta^-})= 0 \,\, \text{a.s. and } \,\, (Y_{\delta^-} - \zeta_{\delta^-})(A'^d_\delta - A'^d_{\delta^-})= 0 \,\, \text{a.s.}, \nonumber
\end{align}
\item[(iii)] $B$ and $B'$ are nondecreasing, right-continuous adapted purely discontinuous processes with $B_{0^-}=B'_{0^-}=0$,
\begin{equation}
(Y_\delta-\xi_\delta)(B_\delta -B_{\delta^-})=0 \,\,\, \text{and} \,\,\, (Y_\delta -\zeta_\delta)(B'_\delta -B'_{\delta^-})=0 \,\, \text{a.s. for all } \,\, \delta \in \mathcal{S}_0,\label{sko2}
\end{equation}
\item[(iv)] \begin{equation} \label{mutualy singular}
dA_t \perp dA'_t \,\, \text{and}\,\, dB_t \perp dB'_t. 
\end{equation} 
\end{enumerate}
%Here $ A^c$ denotes the continuous part of the nondecreasing process $A$ and $A^d$ its discontinuous part. 
Conditions (\ref{sko1}) and (\ref{sko2}) are called minimality conditions or
Skorohod conditions. 
\end{definition}

\begin{remark} 
The Skorohod conditions for $A^d$ and  $A'^d$ (resp. $B$ and $B'$) can be rewritten as 
\begin{align*}
&(Y_{\sigma^-} - \xi_{\sigma^-})(A^d_\sigma - A^d_{\sigma^-}) \mathbf{1}_{G^c}= 0 \,\, \text{a.s. and } \,\, (Y_{\sigma^-} - \zeta_{\sigma^-})(A'^d_\sigma - A'^d_{\sigma^-})\mathbf{1}_{G^c}= 0 \,\, \text{a.s.} \\
& (\text{resp.} \,\,\, (Y_\sigma-\xi_\sigma)(B_\sigma -B_{\sigma^-})\mathbf{1}_{G^c} =0 \,\,\, \text{a.s.} \,\,\,  \text{and} \,\,\, (Y_\sigma -\zeta_\sigma)(B'_\sigma -B'_{\sigma^-})\mathbf{1}_{G^c}=0)  \,\,\, \\ &\text{for all}\,\,\, \delta=(G,\sigma) \in \mathcal{S}_0.
\end{align*}
This follows by the same reasoning applied for $A$ and $B$ in Remark \ref{reee}.
\end{remark}

\begin{remark}
In Definition \ref{def1}, by considering $\delta=(\emptyset, \tau)$, where $\tau$ runs through the
set of stopping times $\mathcal{T}_0$, we recover the usual formulation of the doubly
reflected BSDEs with optional obstacles given in \cite{Doubly Reflected BSDEs and E-Dynkin games}. Moreover, if we associate $T$ with the split terminal time $\rho^T = (\Omega, T )$ and we consider $\delta=(\Omega, \tau)$, where $\tau$ runs through the set of predictable stopping times $\mathcal{T}^p_0$, our definition coincides with the definition of doubly reflected BSDEs for predictable obstacles given in \cite{ArharasBouhadouOuknine}.
%We see that this work generalize the case of optional and predictable obstacles.
\end{remark}

 \begin{remark} \label{Rm}
 Note that a process $(Y,Z,M,A,B,A',B') \in \mathbf{S}^{2} \times \mathbf{H}^2 \times \mathbf{M}^{2,\perp} \times (\mathbf{S}^{2})^2 \times (\mathbf{S}^{2})^2 $ satisfies equation (\ref{eq*}) in the above definition if and only if, almost surely, for all $t$ in $[0,T]$, 
 \begin{multline} \label{eq**}
Y_t=\xi_T + \int_t^T g(s,Y_s,Z_s) ds - \int_t^T Z_s dW_s - (M_{T}-M_{t}) +  A_T-A_t \\- (A'_T-A'_t) +B_{T^-}-B_{t^-}-(B'_{T^-}-B'_{t^-}).
 \end{multline}
 \end{remark}
 
\begin{remark} \label{re}
Note that we have, $- \Delta_+ Y_\delta = \Delta B_\delta - \Delta B'_\delta$ a.s. for all $\delta \in \mathcal{S}_0$. Thus, from condition (\ref{mutualy singular}) we get $\Delta B_\delta= (Y_{\delta^+}-Y_\delta)^-$ a.s. for all $\delta \in \mathcal{S}_0$, and $\Delta B'_\delta= (Y_{\delta^+}-Y_\delta)^+$ a.s. for all $\delta \in \mathcal{S}_0$. In addition, we have
 $- \Delta Y_\delta = \Delta A_\delta - \Delta A'_\delta - \Delta M_\delta$ a.s. The notation $(Y_{\delta^+} - Y_\delta)^-$ and $(Y_{\delta^+} - Y_\delta)^+$ stands for $\max(- (Y_{\delta^+} - Y_\delta),0)$ and $ \max(Y_{\delta^+} - Y_\delta,0)$ respectively.
\end{remark}

\begin{proposition}
Let $g$ be a driver and $(\xi,\zeta)$ be a pair of ladlag admissible obstacles. Let $(Y,Z,M,A,B,A',B') $ be a solution to the doubly reflected BSDE with parameters $(g,\xi,\zeta)$. For each $\delta \in \mathcal{S}_0$, we have 
\begin{equation}\label{vv}
Y_\delta= (Y_{\delta^+} \vee \xi_\delta  )\wedge \zeta_\delta \,\,\, \text{a.s.}
\end{equation}
\end{proposition}
 \begin{proof}
Remark \ref{re} together with the Skorokhod conditions \ref{sko2} for $B$ and $B'$ gives 
\begin{equation*}
(Y_{\delta^+} - Y_\delta)^- = \textbf{1}_{\lbrace Y_\delta = \xi_\delta \rbrace} (Y_{\delta^+} - Y_\delta)^- \,\,\,\text{and} \,\,\, (Y_{\delta^+} - Y_\delta)^+ = \textbf{1}_{\lbrace Y_\delta = \zeta_\delta \rbrace} (Y_{\delta^+} - Y_\delta)^+ \,\,\text{a.s.}
\end{equation*}
It follows easily that (\ref{vv}) holds on the sets $\lbrace \xi_\delta <Y_\delta <\zeta_\delta \rbrace$, $\lbrace \xi_\delta <Y_\delta=\zeta_\delta \rbrace$ and $\lbrace \xi_\delta =Y_\delta<\zeta_\delta \rbrace$, which proves the proposition. 
 \end{proof}
 
 \begin{definition}(\textbf{Mokobodski's condition})
  Let $(\xi,\zeta) \in \mathbf{S}^{2} \times \mathbf{S}^{2}$ be a pair of admissible barriers. The Mokobodski's condition is defined as follows: there exist two nonnegative strong supermartingales $H$ and $\overline{H}$ in $ \mathbf{S}^{2}$ such that:
\begin{equation}
\xi_t \leq H_t- \overline{H}_t \leq \zeta_t \,\,\,\,\,\,\,\, 0\leq t \leq T \,\,\,\,\,\, \text{a.s.}
\end{equation}
 \end{definition}

\subsection{Existence and uniqueness when g is a driver process} \label{sub 4.1}

\hspace*{0.5 cm} In this section, we are given a process $g=(g_t)$ in $\mathbf{H}^2$. Let $\xi$ and $\zeta$ be two ladlag admissible obstacles. We show that the doubly reflected BSDE associated with the driver process $(g_t)$ and the
barriers $\xi$ and $\zeta$ admits a unique solution $(Y, Z, M, A,B, A',B')$.\\

 Let $(Y,Z,M,A,B,A',B') \in \mathbf{S}^{2} \times \mathbf{H}^2 \times \mathbf{M}^{2,\perp} \times (\mathbf{S}^{2})^2 \times (\mathbf{S}^{2})^2 $ be a solution to the DRBSDE associated with driver $g$ and the barriers $\xi$ and $\zeta$.  Let $\tilde{Y}_t := Y_t - \mathbb{E}[\xi_{\rho^T} + \int_t^{\rho^T} g_s ds \vert \mathcal{F}_{t}]$, for all $t\in [0,T]$. From this definition together with equation (\ref{eq*}), we get

\begin{equation}
\tilde{Y}_\delta= J^{g}_\delta- \bar{J}^{g}_\delta \,\,\,\text{a.s.} \,\,\,\, \text{for all}\,\,\,\, \delta  \in \mathcal{S}_0, \label{J-J'}
\end{equation}
where the  processes $J^{g}$ and $\bar{J}^{g}$ are defined, for all $t \in [0,T]$, by 
\begin{equation}
J^{g}_t := \mathbb{E}[A_{\rho^T}-A_t +B_{{\rho^T}^-}- B_{t^-} \vert \mathcal{F}_{t}] \,\,\,\, \text{and} \,\,\,\,
 \bar{J}^{g}_t := \mathbb{E}[A'_{\rho^T} -A'_t +B'_{{\rho^T}^-}- B'_{t^-} \vert \mathcal{F}_{t}].
\end{equation}
Let us denote by $\tilde{\xi}_t^{g}$ and  $\tilde{\zeta}_t^{g}$ the following optional process defined for all $\delta \in\mathcal{S}_0$ by:
 \begin{equation}
 \tilde{\xi}_\delta^{g} := \xi_\delta - \mathbb{E}[\xi_{\rho^T} + \int_\delta^{\rho^T} g_s ds \vert \mathcal{F}_{\delta}], \,\,\,\,\, \tilde{\zeta}_\delta^{g} := \zeta_\delta - \mathbb{E}[\zeta_{\rho^T} + \int_\delta^{\rho^T} g_s ds \vert \mathcal{F}_{\delta}]. \,\,\,\,\, 
 \end{equation}
 
\begin{remark} \label{J,Jbar}
Note that since $A, \,\,A', \,B$ and $B'$ are nondecreasing processes belong to $\mathbf{S}^{2}$, $J^{g}$ and $\bar{J}^{g}$ are two nonnegative strong supermartingales in $\mathbf{S}^{2}$ such that $J^{g}_{\rho^T}=\bar{J}^{g}_{\rho^T}=0$ a.s.
Since $g \in \mathbf{H}^2$ and $\xi , \zeta \in \mathbf{S}^{2}$, it follows that $\tilde{\xi}^{g}$ and $\tilde{\zeta}^{g}$ belong to $\mathbf{S}^{2}$. Moreover, we have $\tilde{\xi}^{g}_{\rho^T}=\tilde{\zeta}^{g}_{\rho^T}=0$ a.s. 
\end{remark}

\begin{lemma} \label{implication}
Let $Y \in \mathbf{S}^{2}$ be the first component of a solution of the doubly RBSDE with parameters $(g,\xi,\zeta)$, where $g$ is a driver and $(\xi, \zeta)$ is a pair of ladlag admissible obstacles. We then have 
$$Y_\delta= J^{g}_\delta - \bar{J}^{g}_\delta + \mathbb{E}[\xi_{\rho^T} + \int_\delta^{\rho^T} g_s ds \vert \mathcal{F}_{\delta}] \,\,\, \text{a.s.}  \,\, \text{for all }\delta \in \mathcal{S}_0,$$ 
where the process $J^{g}$ and $\bar{J}^{g}$ satisfy the following coupled system of reflected BSDEs:
\begin{equation} \label{coupled}
J^{g}= \mathcal{R}ef[(\bar{J}^{g} + \tilde{\xi}^{g} ) \textbf{1}_{[0,T)}]; \,\,\,\, \bar{J}^{g}= \mathcal{P}re[(J^{g}-\tilde{\zeta}^{g} ) \textbf{1}_{[0,T)}],
\end{equation}
where $\mathcal{R}ef$ is the operator associated to the RBSDE over split stopping times with driver 0 (cf. Definition  \ref{operator}).
\end{lemma}

\begin{proof}
From the definition of $\tilde{Y}$ and equality (\ref{J-J'}), it follows that $$Y_\delta= J^{g}_\delta - \bar{J}^{g}_\delta + \mathbb{E}[\xi_{\rho^T} + \int_\delta^{\rho^T} g_s ds \vert \mathcal{F}_{\delta}], \,\, \text{for all }\delta \in \mathcal{S}_0, \,\,\, \text{a.s.},$$ 
By Remark \ref{J,Jbar} the processes $(\bar{J}^{g} + \tilde{\xi}^{g} ) \textbf{1}_{[0,T)}$ and  $(J^{g}-\tilde{\zeta}^{g} ) \textbf{1}_{[0,T)}$ belong to $\mathbf{S}^{2}$. Otherwise, from (i) in Definition \ref{def2} we get for all $\delta \in \mathcal{S}_0$, $\tilde{\xi}^{g}_\delta \leq \tilde{Y}_\delta = J^{g}_\delta - \bar{J}^{g}_\delta \leq \tilde{\zeta}^{g}_\delta$ a.s Therefore, $$J^{g}_\delta \geq \bar{J}^{g}_\delta + \tilde{\xi}^{g}_\delta \,\,\,\, \text{and} \,\,\,\, \bar{J}^{g}_\delta \geq J^{g}_\delta - \tilde{\zeta}^{g}_\delta,\, \,\, \text{a.s. for all}\,\, \delta \in \mathcal{S}_0.$$

From  Remark \ref{J,Jbar}, $J^{g}$ and $\bar{J}^{g}$ are two nonnegative optional strong supermartingales in $\mathbf{S}^{2}$, hence of class $(\mathcal{D})$ (i.e. $\lbrace J^{g}_\tau; \tau \in \mathcal{T}_0 \rbrace$ and $\lbrace \bar{J}^{g}; \tau \in \mathcal{T}_0 \rbrace $ are uniformly integrable). Applying Mertens decomposition for optional strong supermartingales of class $(\mathcal{D})$ (see \cite{Meyer_Un cours sur les integrales stochastiques}, p. 143), we get for all $\delta \in \mathcal{S}_0$
\begin{align}
J^{g}_\delta  = N_{\delta} - A^1_\delta - B^1_{\delta^-}  \,\,\,\text{a.s.;}\,\,\,\,\, \bar{J}^{g}_\delta = \bar{N}_{\delta} - A^2_\delta - B^2_{\delta^-} \,\,\,\text{a.s.}\label{e1}
\end{align}
where; $N$ and $\bar{N}$ are two cadlag uniformly integrable martingales,\\
\hspace*{1.1 cm} $A^1$ and $A^2$ two nondecreasing right-continuous predictable processes with  $A^1_0=$\\ \hspace*{1.1 cm}$A^2_0=0$ and $E[A^1_T]< \infty$, $E[A^2_T]< \infty$, \\
\hspace*{1.1 cm} $B^1$ and $B^2$ two nondecreasing right-continuous adapted purely discontinuous \\
\hspace*{1.1 cm} processes with $B^1_{0^-}=B^2_{0^-}=0$ and $E[B^1_T]< \infty$, $E[B^2_T]< \infty$.\\
Since otherwise, 
\begin{align*}
J^{g}_\delta = \mathbb{E}[A_{\rho^T}+ B_{{\rho^T}^-} \vert \mathcal{F}_{\delta}] -A_\delta - B_{\delta^-}; \\
\bar{J}^{g}_\delta = \mathbb{E}[A'_{\rho^T} +B'_{{\rho^T}^-} \vert \mathcal{F}_{\delta}]-A'_\delta- B'_{\delta^-},
\end{align*}
the uniqueness of Mertens decomposition implies that $A^1 \equiv A$, $B^1 \equiv B$, $A^2 \equiv A'$ and $B^2 \equiv B'$. By the orthogonal decomposition property of martingales in $\mathbf{M}^2$ (Lemma \ref{orthogonal decomposition}), there exist  $(L, M^1), \, (\bar{L}, M^2) \in \mathbf{H}^2 \times \mathbf{M}^{2,\perp} $ such that $\forall \delta \in \mathcal{S}_0$
\begin{align}
 N_\delta= \int_0^\delta L_s dW_s + M^1_\delta \,\,\, \text{a.s.}, 
\bar{N}_\delta= \int_0^\delta \bar{L}_s dW_s + M^2_\delta \,\,\, \text{a.s.} \label{e2}
\end{align}
Therefore, combining equation (\ref{e1}) with (\ref{e2}), we get for all $\delta\in \mathcal{S}_0$
\begin{align}
J^{g}_\delta & = - \int_\delta^{\rho^T} L_s dW_s -( M^1_{{\rho^T}} - M^1_{\delta}) + A_{\rho^T}-A_\delta + B_{{\rho^T}^-} - B_{\delta^-} \,\, \text{a.s.}; \label{eq J^p} \\
\bar{J}^{g}_\delta & = - \int_\delta^{\rho^T} \bar{L}_s dW_s -(M^2_{\rho^T} - M^2_{\delta}) + A'_{\rho^T} - A'_\delta +  B'_{{\rho^T}^-} - B'_{\delta} \,\, \text{a.s.} \label{eq Jbar^p}
\end{align}
Next, we have $Y_\delta -\xi_\delta = \tilde{Y}_\delta - \tilde{\xi}^{g}_\delta = J^{g}_\delta - \bar{J}^{g}_\delta - \tilde{\xi}^{g}_\delta$ a.s. By the Skorokhod conditions (\ref{sko2}), (\ref{sko1}) satisfied by $B$ and $A$, we get 
\begin{align}
\Delta B_\delta [J^{g}_\delta - (\bar{J}^{g}_\delta + \tilde{\xi}^{g}_\delta)]=0 \,\,\,\, \text{a.s.;} \,\,\,\
[J^{g}_{\delta^-} - (\bar{J}^{g}_{\delta^-} + \tilde{\xi}^{g}_{\delta^-})](A^d_\delta - A^d_\delta)=0  \,\,\,\, \text{a.s.} 
\end{align} %%for all } \delta \in \mathcal{T}_0^p.
We also have $\lbrace Y_{t^-} > \xi_{t^-} \rbrace = \lbrace J^{g}_{t^-} > \bar{J}^{g}_{t^-} + \tilde{\xi}^{g}_{t^-} \rbrace$. Hence, the skorokhod condition (\ref{sko1}) satisfied by $A$ can be expressed in the form:
\begin{equation}
\int_0^{\rho^T} \textbf{1}_{ \lbrace J^{g}_{t^-} > \bar{J}^{g}_{t^-} + \tilde{\xi}^{g}_{t^-} \rbrace} dA^c_t =0 \,\,\,\, \text{a.s.}
\end{equation}
We conclude that $(J^{g}, L, M^1, A, B)$ is the solution of the reflected BSDE associated with the driver $0$ and the barrier $(\bar{J}^{g} + \tilde{\xi}^{g})\textbf{1}_{[0,T)}$.\footnote{Note that this barrier is equal to $E(A'_T-A'_t + C'_{T^-} - C'_{t^-} \vert \mathcal{F}_{t}) - \tilde{\xi}^{g}_t$ if $t<T$, and $0 $ if $t=T$. } \\
We prove similarly that  $(\bar{J}^{g}, \bar{L}, M^2, A', B')$ is the solution of the reflected BSDE associated with the driver $0$ and the barrier $(J^{g} - \tilde{\zeta}^{g})\textbf{1}_{[0,T)}$. \footnote{Note that this barrier is equal to $E(A_T-A_t + C_{T^-} - C_{t^-} \vert \mathcal{F}_{t}) - \tilde{\zeta}^{g}_t$ if $t<T$, and $0 $ if $t=T$. } 
This completes the proof. 
\end{proof}

\begin{proposition} \label{equivalent proposition}Let $g \in \mathbf{H}^2$. Let $\xi$ and $\zeta$ be two ladlag admissible obstacles in $\mathbf{S}^{2}$. The following assertions are equivalent:
\begin{enumerate}
\item[(i)] The DRBSDE (\ref{eq*}) with driver process g(t) has a solution. 
\item[(ii)] There exist two processes $J_. \in \mathbf{S}^{2}$ and $\bar{J}_. \in \mathbf{S}^{2}$ satisfying the coupled system of DRBSDEs: 
\begin{equation} \label{coupled*}
J= \mathcal{R}ef[(\bar{J} + \tilde{\xi}^{g} ) \textbf{1}_{[0,T)}]; \,\,\,\,\,\,\, \bar{J}= \mathcal{R}ef[(J -\tilde{\zeta}^{g} ) \textbf{1}_{[0,T)}],
\end{equation}
where $\tilde{\xi}^{g}$ and $ \tilde{\zeta}^{g}$ as above. 
\end{enumerate}
In this matter, the  process $Y$ defined by 
\begin{equation} \label{Y}
Y_\delta := J_\delta - \bar{J}_\delta + E[ \xi_{\rho^T} + \int_\delta^{\rho^T} g_s ds  \vert \mathcal{F}_{\delta}],  \,\,\, \forall \delta \in \mathcal{S}_0 \,\, \text{a.s.}
\end{equation}
gives the first component of the solution to the predictable DRBSDE. 
\end{proposition}

\begin{proof}
The implication $(i) \Rightarrow (ii)$, has been proved in Lemma \ref{implication}. Let us prove $(ii) \Rightarrow (i)$. The steps of the proof are similar of those used in the literature (see eg. \cite{Generalized Dynkin Games and Doubly reflected BSDEs with jumps}, \cite{Doubly Reflected BSDEs and E-Dynkin games}). Let $(J, L, M^1,A,B)$ and $(\bar{J}, \bar{L}, M^2,A',B') $  be two solutions in $ \mathbf{S}^{2} \times \mathbf{H}^2 \times \mathbf{M}^{2,\perp} \times (\mathbf{S}^{2})^2 \times (\mathbf{S}^{2})^2$ of the coupled system (\ref{coupled*}). For all $\delta \in \mathcal{S}_0$, we define the process $Y$ as in (\ref{Y}). \\
\hspace*{0.5 cm} By assumptions, the processes $J$ and $\bar{J}$ belong to $\mathbf{S}^{2}$. Hence, the difference $J - \bar{J}$ and thus the process $Y$ are well defined. Further, since $J_{\rho^T}=\bar{J}_{\rho^T}=0$ a.s., we get $Y_{\rho^T}=\xi_{\rho^T}$ a.s. By the formulation of the coupled system (\ref{coupled*}), we get $J_\delta \geq \bar{J}_\delta + \tilde{\xi}^{g}_\delta$ and $\bar{J}_\delta \geq J_\delta - \tilde{\zeta}^{g}_\delta$ a.s. for all $\delta \in \mathcal{S}_0$. Then, we derive that $ \xi_\delta \leq Y_\delta \leq \zeta_\delta$ a.s. for all $\delta \in\mathcal{S}_0$. \\
 \hspace*{0.5 cm}
 Note that the process $(E(\xi_{\rho^T} + \int_t^{\rho^T} g_s \vert \mathcal{F}_{t}))_{t\in [0,T]}$ coincides with the first component of the solution to the (non-reflected) BSDE with terminal condition $\xi_{\rho^T}$ and driver $g$ (cf. \cite{Optimal stopping with}, Definition 2.2). Thus, there exist $(L', \bar{M}) \in \mathbf{H}^2 \times \mathbf{M}^{2,\perp}$ such that for all $\delta \in \mathcal{S}_0 $:
 \begin{equation}
 E(\xi_{\rho^T} + \int_\delta^{\rho^T} g_s ds \vert \mathcal{F}_{\delta}) = \xi_{\rho^T} + \int_\delta^{\rho^T} g_s ds - \int_\delta^{\rho^T} L'_s dW_s - (\bar{M}_{{\rho^T}} - \bar{M}_{\delta}) \,\,\, \text{a.s.}.
\end{equation}  
From this, together with (\ref{Y}) and the definition of $J$ and $\bar{J}$, we get  
\begin{multline}
Y_\delta= \xi_{\rho^T} + \int_\delta^{\rho^T} g_s ds - \int_\delta^{\rho^T} Z_s dW_s -( M_{{\rho^T}}- M_{\delta}) + A_{\rho^T}- A_\delta - (A'_{\rho^T}- A'_\delta) \\ + B_{{\rho^T}^-} - B_{\delta^-}  - (B'_{{\rho^T}^-} - B'_{\delta^-}),
\end{multline}
where $Z:= L-\bar{L} + L'$, $M=M^1-M^2 + \bar{M}$. \\
%(resp. $\int_0^T \textbf{1}_{\lbrace \bar{J}^p_{t^-} > J^p_{t^-} - \tilde{\zeta}^{g,p}_{t^-} \rbrace} dA'_t=0$ a.s. and for all $\tau \in \mathcal{T}^p_0$, $\Delta B'_{\tau} = \textbf{1}_{\lbrace \bar {J}^p_\tau = J^p_\tau - \tilde{\zeta}^{g,p}_\tau \rbrace} \Delta B_{\tau}$ a.s.), 
Skorokhod conditions (\ref{sko1}) and (\ref{sko2}) for $A$, $A'$, $B$ and $B'$: From the above, it follows that the processes $A$, $B$ (resp. $A'$, $B'$) satisfy the Skorokhod conditions for RBSDEs. Accordingly, we have: for all $\delta \in \mathcal{S}_0$, $\Delta A^d_\delta = \textbf{1}_{\lbrace J_{\delta^-} = \bar{J}_{\delta^-} + \tilde{\xi}^{g}_{\delta^-}\rbrace} \Delta A^d_\tau=$ a.s.; $\Delta B_{\delta} = \textbf{1}_{\lbrace J_\delta = \bar{J}_\delta + \tilde{\xi}^{g}_\tau \rbrace} \Delta B_{\delta}$ a.s. and $\int_0^{\rho^T} \textbf{1}_{\lbrace J_{t^-} > \bar{J}_{t^-} + \tilde{\xi}^{g}_{t^-} \rbrace} dA^c_t=0$ a.s. Similar conditions hold for $A'$ and $B'$. By the definition of $Y$, $\tilde{\xi}^{g}$ and $\tilde{\zeta}^{g}$, we get 
\begin{align*}
\lbrace J_{\delta^-} > \bar{J}_{\delta^-} + \tilde{\xi}^{g}_{\delta^-} \rbrace = &\lbrace Y_{\delta^-} > \xi_{\delta^-} \rbrace, \,\,\,\,
 \lbrace J_\delta = \bar{J}_\delta + \tilde{\xi}^{g}_\delta \rbrace = \lbrace Y_\delta=\xi_\delta \rbrace,\\ 
  &\lbrace J_{\delta^-} = \bar{J}_{\delta^-} + \tilde{\xi}^{g}_{\delta^-}\rbrace = \lbrace Y_{\delta^-} = \xi_{\delta^-} \rbrace. 
\end{align*}
Thus, we get (\ref{sko1}) and (\ref{sko2}) for $A$ and $B$. By the same manner, we prove the Skorokhod conditions for $A'$ and $B'$. The only point remaining concerns the behavior of the random measures $dA_t$, $dA'_t$ and $dB_t$, $dB'_t$. If $dA_t \perp dA'_t$ and $dB_t \perp  dB'_t$, then the vector $(Y,Z,M,A,B,A',B')$ is a solution to the doubly RBSDE with parameters $(g,\xi,\zeta)$, which is the desired conclusion. If not, by the canonical decomposition of a RCLL predictable process with integrable total variation (see, Proposition A.7 in \cite{Generalized Dynkin Games and Doubly reflected BSDEs with jumps}, p. 30), there exist a pair $(C,C')$ (resp. $(D,D')$) of nondecreasing right-continuous predictable processes belonging in $\mathbf{S}^{2}$, such that $A-A'= C-C'$ (resp. $B-B'= D-D'$) with $dC_t \perp dC'_t$ (resp. $dD_t \perp dD'_t$). Moreover, $dC_t \ll dA_t$, $dC'_t \ll dA'_t$, $dD_t \ll dB_t$ and $dD'_t \ll dB'_t$. Then since $\int_0^{\rho^T} \textbf{1}_{\lbrace Y_{t^-} > \xi_{t^-} \rbrace} dA^c_t =0$ a.s., we get $\int_0^{\rho^T} \textbf{1}_{\lbrace Y_{t^-} > \xi_{t^-} \rbrace} dC^c_t =0$ a.s. Also, for all $\delta=(G,\sigma) \in \mathcal{S}_0$ on $G^c$ we have $\Delta A^d_\sigma = \mathbf{1}_{\lbrace Y_{\sigma^-} = \xi_{\sigma^-} \rbrace} \Delta A^d_\sigma $ a.s., thus $(Y_{\sigma^-} - \xi_{\sigma^-})(C^d_\sigma - C^d_{\sigma^-}) \mathbf{1}_{G^c}= 0$ a.s. for all $\delta=(G,\sigma) \in \mathcal{S}_0$. Similarly, we obtain $\int_0^{\rho^T} \textbf{1}_{\lbrace Y_{t^-} < \zeta_{t^-} \rbrace} dC'^{c}_t =0$ and $(Y_{\sigma^-} - \zeta_{\sigma^-})(C'^d_\sigma - C'^d_{\sigma^-}) \mathbf{1}_{G^c}= 0$ a.s. for all $\delta=(G,\sigma) \in \mathcal{S}_0$, hence the processes $C$ and $C'$ satisfy the Skorokhod conditions (\ref{sko1}).\\ Therefore, the observation $dD_t \ll dB_t$ implies that $D$ is purely discontinuous and for all $\delta=(G,\sigma) \in \mathcal{S}_0$, $(Y_\sigma - \xi_\sigma) \Delta D_\sigma \textbf{1}_{H^c} = 0$ a.s. Similarly, $D'$ is purely discontinuous and $(Y_\sigma - \zeta_\sigma) \Delta D'_\sigma \textbf{1}_{H^c} = 0$ a.s. for all $\delta=(G,\sigma) \in \mathcal{S}_0$. Hence, the processes $D$ and $D'$ satisfy the Skorokhod conditions (\ref{sko2}).  We conclude that the process $(Y,Z,M,C,D,C',D')$ is a solution to the doubly RBSDE with parameters $(f,\xi,\zeta)$. Which completes the proof.\\
\end{proof}

%\hspace*{0.5 cm} From Proposition \ref{equivalent proposition}, the existence of a solution of the coupled system (\ref{coupled*}) imply the existence of a solution to the DRBSDE associated with driver process $(g_t)$. In the following lemma we prove that, under Mokobodzki's condition, there exist a solution of the coupled system of reflected BSDEs (\ref{coupled*}). \\
%iterative approximation to the coupled system

Set now $\mathbf{J}^{0}_.=0$ and $\bar{\mathbf{J}}^{0}_.=0$ . We define recursively, for each $n\in \mathbb{N}$, the processes:\footnote{We omit the exponent $g$ in the notation for $\mathbf{J}^{n}$ and $\bar{\mathbf{J}}^{n}$ for sake of simplicity. \label{refnote}}
\begin{equation} \label{coupled**}
\mathbf{J}^{n+1} := \mathcal{R}ef[(\bar{\mathbf{J}}^{n} + \tilde{\xi}^{g} ) \textbf{1}_{[0,T)}]; \,\,\,\,\,\,\, \bar{\mathbf{J}}^{n+1} := \mathcal{R}ef[(\mathbf{J}^{n}  -\tilde{\zeta}^{p} ) \textbf{1}_{[0,T)}].
\end{equation}
Since $\tilde{\xi}^{g}$ and $ \tilde{\zeta}^{g}$ belong to $ \mathbf{S}^{2}$, by induction, the processes $\mathbf{J}^{n}$ and $\bar{\mathbf{J}}^{n}$ are well defined.

\begin{lemma}\label{lim}
Let $(\xi,\zeta) \in \mathbf{S}^{2} \times \mathbf{S}^{2}$ be a pair of ladlag admissible barriers satisfies Mokobodzki's condition. The sequences of processes $(\mathbf{J}^{n}_.)_{n\in\mathbb{N}}$ and $(\bar{\mathbf{J}}^{n}_.)_{n\in\mathbb{N}}$ are nondecreasing. The processes $\mathbf{J}_.$ and $\bar{\mathbf{J}}_.$ defined by 
\begin{equation} 
\mathbf{J}_. := \lim_{n \rightarrow +\infty} \mathbf{J}^{n}_. \,\,\, \text{and} \,\,\, \bar{\mathbf{J}}_. := \lim_{n \rightarrow +\infty} \bar{\mathbf{J}}^{n}_.
\end{equation}
 are the smallest nonnegative strong supermartingales in $\mathbf{S}^{2}$ satisfying the system (\ref{coupled*}). We also have following minimality property: if $H$ and $\bar{H}$ are two nonnegative strong supermartingales such that $\tilde{\xi}^{g} \leq H-\bar{H} \leq \tilde{\zeta}^{g}$, then we have $\mathbf{J} \leq H$ and $\bar{\mathbf{J}} \leq \bar{H}$. 
\end{lemma}

\begin{proof}
Using Lemma \ref{incr opera} and Remark \ref{cv seq of p.s.s}, the proof runs in much the same way as in \cite[Proposition 3.14]{Doubly Reflected BSDEs and E-Dynkin games}. 
\end{proof}

From Lemma \ref{lim} together with Proposition \ref{equivalent proposition}, we derive the following existence result.

\begin{theorem}(\textbf{Existence of a solution to the DRBSDE}) \label{existence}
Let $g=(g_t)\in \mathbf{H}^{2}$ be a driver process. Let $(\xi,\zeta)$ be a pair of ladlag admissible barriers satisfying Mokobodzki's condition. Then, there exists a solution of the doubly RBSDE associated with the driver $g$. The first component of this solution is given by 
\begin{equation}
Y_\delta := \mathbf{J}_\delta - \bar{\mathbf{J}}_\delta + E[ \xi_{\rho^T} + \int_\delta^{\rho^T} g_s ds  \vert \mathcal{F}_{\delta}] \,\,\,\, \text{a.s.} \,\,\,\, \text{for all}\,\,\, \delta \in \mathcal{S}_0.
\end{equation}
where $ \mathbf{J}$ and $\bar{\mathbf{J}}$ are the processes defined in Lemma \ref{lim}.
\end{theorem}

Let us now investigate the question of uniqueness of the solution to the doubly
RBSDE defined above. To this purpose, we state the following lemma.

\begin{lemma}(\textbf{A priori estimate}) \label{lemma uniqueness} Let $(Y,Z,M,A,B,A',B') \in \mathbf{S}^{2} \times \mathbf{H}^2 \times \mathbf{M}^{2,\perp} \times (\mathbf{S}^{2})^2 \times (\mathbf{S}^{2})^2$ (resp. $(\bar{Y}, \bar{Z}, \bar{M}, \bar{A}, \bar{B}, \bar{A}',\bar{B}') \in \mathbf{S}^{2} \times \mathbf{H}^2 \times \mathbf{M}^{2,\perp} \times (\mathbf{S}^{2})^2 \times (\mathbf{S}^{2})^2$) be a solution to the DRBSDE associated with driver $g=(g_t) \in \mathbf{H}^{2}$ (resp. $\bar{g}=(\bar{g}_t) \in \mathbf{H}^{2}$) and the ladlag admissible barriers $\xi$ and $\zeta$. Then, there exists $c>0$ such that for all $\epsilon \geq 0$, for all $\beta >\frac{1}{\epsilon^2}$, we have 
\begin{equation}\label{ll1}
\Vert Z- \bar{Z} \Vert^2_{\beta} + \Vert M- \bar{M} \Vert^2_{\mathbf{M}^2_\beta} \leq \epsilon^2 \Vert g- \bar{g} \Vert^2_{\beta}; \\
\end{equation}
\begin{equation} \label{ll11}
\vertiii{Y- \bar{Y}}^2_{\beta} \leq 4\epsilon^2(1 + 6c^2) \Vert g- \bar{g} \Vert^2_{\beta}.
\end{equation}
\end{lemma}

\begin{proof} 
The proof of this result in the case of ${\cal E}^{f}$-Dynkin games over stopping times can be found
in \cite[Lemma 3.18]{Doubly Reflected BSDEs and E-Dynkin games}. The proof in our case is identical and is
therefore omitted.
\end{proof}

From this result, we derive the following uniqueness result for the DRBSDE associated with the driver process $g(t)$.

\begin{theorem}
Let $\xi$ and $\zeta$ be two ladlag admissible barriers satisfying Mokobodzki's condition. The DRBSDE (\ref{eq*}) associated with driver process $g(t)$ admits a unique solution $(Y,Z,M,A,B,A',B') \in \mathbf{S}^{2} \times \mathbf{H}^2 \times \mathbf{M}^{2,\perp} \times (\mathbf{S}^{2})^2 \times (\mathbf{S}^{2})^2$. 
\end{theorem}

\begin{proof}
We only need to show the uniqueness of the solution, Theorem \ref{existence} gives the existence. For this purpose, let $(Y,Z,M,A,B,A',B')$ be a solution of the DRBSDE associated with the driver process $g(t)$ and the obstacles $(\xi, \zeta)$. The previous estimates (\ref{ll1}) and (\ref{ll11}) (applied with $g=\bar{g}$), gives the uniqueness of $(Y, Z, M)$. Thus, by Remark \ref{re} $B$ and $B'$ are uniquely determined and by using (\ref{eq*}) and the mutually singularity condition (\ref{mutualy singular}), we get the uniqueness of $A$ and $A'$, which completes the proof. 
\end{proof}

\subsection{Existence and uniqueness for a general Lipschitz driver} \label{sub 4.2}
\hspace*{0.5 cm} In this section, we are given a Lipschitz driver $g$. We prove existence and uniqueness of the solution to the DRBSDE from Definition \ref{def2}, in the case of a general Lipschitz driver $g$. 

\begin{theorem}
Let $\xi$ and $\zeta$ be two ladlag admissible barriers satisfying Mokobodzki's condition and let $g$ be a Lipschitz driver. There exists a unique solution to the DRBSDE (\ref{eq*}) associated with parameters $(g,\xi,\zeta)$.
\end{theorem}
\begin{proof} 
For $\beta >0$, we introduce a mapping $\Phi$ from $\mathbf{K}^2_\beta$ into itself. This map is defined by: for a given $(U,V)\in \mathbf{K}^2_\beta$, $\Phi(U,V) := (Y,Z)$, where $Y$, $Z$ are the first two components of the solution $(Y,Z,M,A,B,A',B')$ to the DRBSDE associated with driver $g_t:= g(t,U_t,V_t)$ and with the pair of ladlag admissible barriers $(\xi,\zeta)$. Note that by Theorem \ref{existence}, the mapping $\Phi$ is well-defined. Using Lemma \ref{lemma uniqueness} and analysis similar to that in the proof of Theorem \ref{uniqueness}, we conclude that $\Phi$ is a contraction, and hence by the Banach fixed-point theorem, admits a unique fixed point $(Y,Z) \in \mathbf{K}_\beta^2$. By the definition of $\Phi$, the process $(Y,Z)$ is equal to the first two components of the solution $(Y,Z,M,A,B,A',B')$ to the DRBSDE associated with the driver process $h(\omega,t):= g(\omega,t,Y_t(\omega), Z_t(\omega))$ and with the pair of barriers $(\xi,\zeta)$. Thus, $(Y,Z,M,A,B,A',B')$ is the unique solution to the DRBSDE  with parameters $(g,\xi,\zeta)$, which completes the proof. 
\end{proof}

% \section{Data Availability Statement}

% \hspace*{0.5 cm} No data-sets were generated or analyzed during the current study.

% \section*{Acknowledgements}

% \hspace*{0.5 cm} The authors thank the anonymous Referee for his valuable comments and suggestions from which the manuscript greatly benefited.


\begin{thebibliography}{9}

\small

\bibitem{Nikeghbali} A. Nikeghbali, An essay on the general theory of stochastic processes. Probab. Surv. 3 (2006), 345--412. \hyperlink{https://mathscinet.ams.org/mathscinet-getitem?mr=MR2280298}{MR2280298} \vspace*{-0.2 cm}


\bibitem{baadi1} B. Baadi, Y. Ouknine, Reflected BSDEs when the obstacle is not right-continuous in a general filtration. ALEA Lat. Am. J. Probab. Math. Stat. 14 (2017), no. 1, 201--218. \hyperlink{https://mathscinet.ams.org/mathscinet-getitem?mr=MR3655127}{MR3655127} \vspace*{-0.2cm}

\bibitem{baadi2} B. Baadi, Y. Ouknine, Reflected BSDEs with optional barrier in a general filtration. Afr. Mat. 29 (2018), no. 7-8, 1049--1064. \hyperlink{https://mathscinet.ams.org/mathscinet-getitem?mr=MR3869547}{MR3869547} \vspace*{-0.2 cm}

\bibitem{El Asri and al.} B. El Asri, S. Hamadene, H. Wang, $L^p$-solutions for doubly reflected backward stochastic differential equations. Stoch. Anal. Appl. 29 (2011), no. 6, 907--932. \hyperlink{https://mathscinet.ams.org/mathscinet-getitem?mr=MR2847329}{MR2847329} \vspace*{-0.2 cm}

\bibitem{Probabilites et Potentiel2} C. Dellacherie and P.-A. Meyer, Probabilités et Potentiel, Théorie des Martingales. (French) Chap. V-VIII. Nouvelle édition. Hermann, Paris, 1980.  \hyperlink{https://mathscinet.ams.org/mathscinet-getitem?mr=MR0566768}{MR0566768} \vspace*{-0.2 cm}

\bibitem{Probabilites et Potentiel1} C. Dellacherie and P.-A. Meyer, Probabilités et Potentiel. (French) Chap. I-IV. Nouvelle édition. Hermann,  Paris, 1975. \hyperlink{https://mathscinet.ams.org/mathscinet-getitem?mr=MR0488194}{MR0488194} \vspace*{-0.2 cm}


\bibitem{Essaky-Harraj-Ouknine} E. H. Essaky, N. Harraj, Y. Ouknine, Backward stochastic differential equations with two reflecting barriers and jumps. Stoch. Anal. Appl. 23 (2005), no. 5, 921--938. \hyperlink{https://mathscinet.ams.org/mathscinet-getitem?mr=MR2158885}{MR2158885} \vspace*{-0.2 cm}


\bibitem{Essaky(2008)} E. H. Essaky, Reflected backward stochastic differential equation with jumps and RCLL obstacle. Bull. Sci. Math. 132 (2008), no. 8, 690--710.  \hyperlink{https://mathscinet.ams.org/mathscinet-getitem?mr=MR2474488}{MR2474488} \vspace*{-0.2 cm}

\bibitem{Lenglart} E. Lenglart, Tribus de Meyer et théorie des processus, (French) Seminar on Probability, XIV (Paris, 1978/1979) (French), pp. 500--546, Lecture Notes in Math., 784, Springer, Berlin, 1980. \hyperlink{https://mathscinet.ams.org/mathscinet-getitem?mr=MR0580151}{MR0580151} \vspace*{-0.2 cm}

\bibitem{Pardoux and Peng} E. Pardoux, S. Peng, Adapted solution of a backward stochastic differential equation. Systems Control Lett. 14 (1990), no. 1, 55--61. \hyperlink{https://mathscinet.ams.org/mathscinet-getitem?mr=MR1037747}{MR1037747} \vspace*{-0.2 cm}

\bibitem{ArharasBouhadouOuknine} I. Arharas, S. Bouhadou, Y. Ouknine, Doubly Reflected Backward Stochastic Differential Equations in the Predictable Setting,  J. Theoret. Probab. (2021), 1–27.  \hyperlink{https://link.springer.com/article/10.1007/s10959-020-01070-5}{DOI: 10.1007/s10959-020-01070-5}
\vspace*{-0.2 cm}

\bibitem{Kara} I. Karatzas, S. Shreve, S. E., Brownian Motion and Stochastic Calculus, 2nd ed. Graduate Texts in Mathematics 113. Springer, New York  (1994). \hyperlink{https://mathscinet.ams.org/mathscinet-getitem?mr=1121940}{MR1121940} \vspace*{-0.2 cm}

\bibitem{Cvitanic and Karatzas} J. Cvitanic, I. Karatzas, Backward stochastic differential equations with reflection and Dynkin games. Ann. Probab. 24 (1996), no. 4, 2024--2056. \hyperlink{https://mathscinet.ams.org/mathscinet-getitem?mr=MR1415239}{MR1415239} \vspace*{-0.2 cm}

\bibitem{Jacod and Shiryaev} J. Jacod, A. N. Shiryaev. Limit theorems for stochastic processes, volume 288 of Grundlehren der Mathematischen Wissenschaften [Fundamental Principles of Mathematical Sciences]. Springer-Verlag, Berlin, second edition, 2003. \hyperlink{https://mathscinet.ams.org/mathscinet-getitem?mr=MR1943877}{MR1943877} \vspace*{-0.2 cm}

\bibitem{Bismut} J.M. Bismut, Conjugate convex functions in optimal stochastic control. J. Math. Anal. Appl. 44 (1973), 384--404. \hyperlink{https://mathscinet.ams.org/mathscinet-getitem?mr=MR0329726}{MR0329726} \vspace*{-0.2 cm}

\bibitem{Bismut1} J.M. Bismut, Temps d'arrêt optimal, quasi-temps d'arrêt et retournement du temps. (French) Ann. Probab. 7 (1979), no. 6, 933–964. \hyperlink{https://mathscinet.ams.org/mathscinet/search/publdoc.html?pg1=INDI&s1=191108&sort=Newest&vfpref=html&r=146&mx-pid=548890}{MR0548890} \vspace*{-0.2 cm}

\bibitem{Neveu} J. Neveu, Martingales à Temps Discret. Masson et Cie, éditeurs, Paris, 1972. \hyperlink{https://mathscinet.ams.org/mathscinet-getitem?mr=MR0402914}{MR0402914} \vspace*{-0.2 cm}

\bibitem{Gal'chouk} L.I. Gal'chouk, Decomposition of optional supermartingales. (Russian) Mat. Sb. (N.S.) 115(157) (1981), no. 2, 163–178, 319.  \hyperlink{https://mathscinet.ams.org/mathscinet/search/publdoc.html?pg1=INDI&s1=201289&sort=Newest&vfpref=html&r=43&mx-pid=609171}{MR0609171} \vspace*{-0.2 cm}

\bibitem{DG}{M. Alario-Nazaret, J.P. Lepeltier, B. Marchal: Dynkin games. Lect. Notes Control Inf. Sci. 43. Springer, Berlin, 1982, pp. 23–32). \hyperlink{https://mathscinet.ams.org/mathscinet-getitem?mr=814103}{MR0814103}}

\bibitem{Maingueneau} M.-A. Maingueneau, Temps d'arrêt optimaux et théorie générale, Séminaire de probabilités, XII (Univ. Strasbourg, Strasbourg, 1976/1977), pp. 457--467, Lecture Notes in Math., 649, Springer, Berlin, 1978. \hyperlink{https://mathscinet.ams.org/mathscinet-getitem?mr=MR0520020}{MR0520020} \vspace*{-0.2 cm}

\bibitem{Quenez-Sulem (2014)}  M.-C. Quenez, A. Sulem, Reflected BSDEs and robust optimal stopping for dynamic risk measures with jumps. Stochastic Process. Appl. 124 (2014), no. 9, 3031--3054. \hyperlink{https://mathscinet.ams.org/mathscinet-getitem?mr=MR3217432}{MR3217432} \vspace*{-0.2 cm}

\bibitem{Reflected BSDEs when the obstacle is not right-continuous and optimal stopping} M. Grigorova, P. Imkeller, E. Offen, Y. Ouknine, M.-C. Quenez, Reflected BSDEs when the obstacle is not right-continuous and optimal stopping. Ann. Appl. Probab. 27 (2017), no. 5, 3153--3188. \hyperlink{https://mathscinet.ams.org/mathscinet-getitem?mr=MR3719955}{MR3719955} \vspace*{-0.2 cm}

\bibitem{Doubly Reflected BSDEs and E-Dynkin games} M. Grigorova, P. Imkeller, Y. Ouknine, M.-C. Quenez, Doubly Reflected BSDEs and $\varepsilon^f$-Dynkin games: beyond the right-continuous case.  Electron. J. Probab. 23 (2018), Paper No. 122, 38 pp. \hyperlink{https://mathscinet.ams.org/mathscinet-getitem?mr=MR3896859}{MR3896859} \vspace*{-0.2 cm}

\bibitem{Optimal stopping with} M. Grigorova, P. Imkeller, Y. Ouknine, M.-C. Quenez, Optimal stopping with $f$- expectations: the irregular case. Stochastic Process. Appl. 130 (2020), no. 3, 1258--1288. \hyperlink{https://mathscinet.ams.org/mathscinet-getitem?mr=MR4058273}{MR4058273} \vspace*{-0.2 cm}

\bibitem{Kobylanski} M. Kobylanski, M.-C. Quenez, E. Rouy-Mironescu, Optimal multiple stopping time problem. Ann. Appl. Probab. 21 (2011), no. 4, 1365--1399. \hyperlink{https://mathscinet.ams.org/mathscinet-getitem?mr=MR2857451}{MR2857451} \vspace*{-0.2 cm}

\bibitem{KobylanskiQuenezRoger} M. Kobylanski, M.-C. Quenez, M. Roger de Campagnolle, Dynkin games in a general framework. Stochastics 86 (2014), no. 2, 304--329.  \hyperlink{https://mathscinet.ams.org/mathscinet-getitem?mr=MR3180039}{MR3180039} \vspace*{-0.2 cm}

\bibitem{KobylanskiQuenez} M. Kobylanski, M.-C. Quenez, Optimal stopping time problem in a general framework. Electron. J. Probab. 17 (2012), no. 72, 28 pp. \hyperlink{https://mathscinet.ams.org/mathscinet-getitem?mr=MR2968679}{MR2968679} \vspace*{-0.2 cm}


\bibitem{Marzougue} M. Marzougue, A note on optional Snell envelopes and reflected backward SDEs. Statist. Probab. Lett. 165 (2020), 108833, 7 pp. \hyperlink{https://mathscinet.ams.org/mathscinet-getitem?mr=MR4113844}{MR4113844} \vspace*{-0.2 cm}


\bibitem{El Karoui and al.} N. El Karoui, C. Kapoudjian, E. Pardoux, S. Peng and M.C. Quenez, Reflected solutions of backward SDE's, and related obstacle problems for PDE's. Ann. Probab. 25 (1997), no. 2, 702--737. \hyperlink{https://mathscinet.ams.org/mathscinet-getitem?mr=MR1434123}{MR1434123} 
\vspace*{-0.2 cm}

\bibitem{El Karoui} N. El Karoui, Les aspects probabilistes du contrôle stochastique. École d’été de Probabilités de Saint-Flour IX-1979 Lect. Notes in Math. 876, 73-238, Springer, Berlin-New York, 1981. \hyperlink{https://mathscinet.ams.org/mathscinet-getitem?mr=MR0637471}{MR0637471} 
\vspace*{-0.2 cm}

\bibitem{Meyer_Un cours sur les integrales stochastiques} P.-A. Meyer, Un cours sur les intégrales stochastiques. (French) Séminaire de Probabilités, X (Seconde partie: Théorie des intégrales stochastiques, Univ. Strasbourg, Strasbourg, année universitaire 1974/1975), pp. 245–400. Lecture Notes in Math., Vol. 511, Springer, Berlin, 1976. \hyperlink{https://mathscinet.ams.org/mathscinet-getitem?mr=MR0501332}{MR0501332} \vspace*{-0.2 cm}

\bibitem{Protter} P.E. Protter, Stochastic integration and differential equations. Second edition. Version 2.1. Corrected third printing. Stochastic Modelling and Applied Probability, 21. Springer-Verlag, Berlin, 2005. \hyperlink{https://mathscinet.ams.org/mathscinet-getitem?mr=MR2273672}{MR2273672}
\vspace*{-0.2 cm}

\bibitem{Game Options in an Imperfect Market with Default}  R. Dumitrescu, M.-C. Quenez, A. Sulem, Game options in an imperfect market with default. SIAM J. Financial Math. 8 (2017), no. 1, 532--559. \hyperlink{https://mathscinet.ams.org/mathscinet-getitem?mr=MR3679314}{MR3679314} \vspace*{-0.2 cm}

\bibitem{Generalized Dynkin Games and Doubly reflected BSDEs with jumps} R. Dumitrescu, M.-C. Quenez, A. Sulem, Generalized Dynkin games and doubly reflected BSDEs with jumps. Electron. J. Probab. 21 (2016), Paper No. 64, 32 pp. \hyperlink{https://mathscinet.ams.org/mathscinet-getitem?mr=MR3580030}{MR3580030} \vspace*{-0.2 cm}

\bibitem{Bouhadou} S. Bouhadou, Y. Ouknine, Reflected BSDEs when the obstacle is predictable and nonlinear optimal stopping problem,  Stoch. Dyn. (2021). \hyperlink{https://www.worldscientific.com/doi/abs/10.1142/S0219493721500490?download=true&journalCode=sd}{DOI: 10.1142/S0219493721500490} \vspace*{-0.2 cm}


\bibitem{Crepey-Matoussi} S. Crépey, A. Matoussi, Reflected and doubly reflected BSDEs with jumps: a priori estimates and comparison. Ann. Appl. Probab. 18 (2008), no. 5, 2041--2069. \hyperlink{https://mathscinet.ams.org/mathscinet-getitem?mr=MR2462558}{MR2462558} \vspace*{-0.2 cm}

\bibitem{Hamadene-Hassani-Ouknine} S. Hamadène, M. Hassani, Y. Ouknine, Backward SDEs with two $rcll$ reflecting barriers without Mokobodski's hypothesis. Bull. Sci. Math. 134 (2010), no. 8, 874--899. \hyperlink{https://mathscinet.ams.org/mathscinet-getitem?mr=MR2737357}{MR2737357} \vspace*{-0.2 cm}

\bibitem{Hamadene} S. Hamadène, Reflected BSDE's with discontinuous barrier and application. Stoch. Stoch. Rep. 74 (2002), no. 3-4, 571--596. \hyperlink{https://mathscinet.ams.org/mathscinet-getitem?mr=MR1943580}{MR1943580} \vspace*{-0.2 cm}


\bibitem{Hamadene-Ouknine} S. Hamadène, Y. Ouknine, Reflected backward SDEs with general jumps. Theory Probab. Appl. 60 (2016), no. 2, 263--280. \hyperlink{https://mathscinet.ams.org/mathscinet-getitem?mr=MR3568776}{MR3568776} \vspace*{-0.2 cm}


\bibitem{Hamadene-Ouknine2003} S. Hamadène, Y. Ouknine, Reflected backward stochastic differential equation with jumps and random obstacle. Electron. J. Probab. 8 (2003), no. 2, 20 pp. \hyperlink{https://mathscinet.ams.org/mathscinet-getitem?mr=MR1961164}{MR1961164} 


 
\end{thebibliography}
\end{document}